\documentclass[10pt]{amsart}
\usepackage[a4paper, margin=1in, marginparwidth=0.7cm]{geometry}
\usepackage{amssymb}
\usepackage[centertags]{amsmath}
\usepackage{amsthm}
\usepackage{amsfonts}
\usepackage{eucal}
\usepackage{mathrsfs}

\usepackage[all,cmtip]{xy}
\usepackage{graphicx, eepic}

\usepackage{fancyhdr}

\usepackage[shortlabels]{enumitem} %
\usepackage{lmodern}
\usepackage[utf8]{inputenc}
\usepackage[T1]{fontenc}
\usepackage{xcolor}

\usepackage{bbm}
\usepackage{pdfpages}
\usepackage{titletoc}
\usepackage{booktabs}
\usepackage{mathtools}
\usepackage{thmtools}

\usepackage{epstopdf}
\usepackage{epsfig}
\usepackage{microtype}

\definecolor{blue1}{rgb}{0,0,0.8}
\definecolor{hanul}{rgb}{0,0.5,0.8}
\usepackage[colorlinks=true, citecolor=hanul, linkcolor=magenta, urlcolor=blue1, filecolor=cyan, 
]{hyperref}

\usepackage{multirow, url}
\usepackage{textcomp}

\DeclareMathAlphabet{\cmcal}{OMS}{cmsy}{m}{n}
\newtheorem{thm}{Theorem}[section]
\newtheorem{cor}[thm]{Corollary}
\newtheorem{lem}[thm]{Lemma}
\newtheorem{prop}[thm]{Proposition}

\newtheorem{conj}[thm]{Conjecture}

\theoremstyle{definition}
\newtheorem{defn}[thm]{Definition}

\theoremstyle{remark}
\newtheorem{rem}[thm]{\bf{Remark}}

\numberwithin{equation}{section} \numberwithin{table}{section}
	
\newtheorem*{thm*}{\bf{Theorem}}
\newtheorem*{claim*}{\bf{Claim}}
\newtheorem*{rem*}{\bf{Remark}}
\newtheorem*{rems*}{\bf{Remarks}}
\newtheorem*{exam*}{\bf{Example}}
\newtheorem*{exams*}{\bf{Examples}}

\newcommand{\F}{{\mathbb{F}}}

\newcommand{\bP}{{\mathbb{P}}}
\newcommand{\Q}{{\mathbb{Q}}}

\newcommand{\T}{{\mathbb{T}}}

\newcommand{\Z}{{\mathbb{Z}}}

\newcommand{\fa}{{\mathfrak{a}}}
\newcommand{\fb}{{\mathfrak{b}}}

\newcommand{\fh}{{\mathfrak{h}}}

\newcommand{\fm}{{\mathfrak{m}}}
\newcommand{\fn}{{\mathfrak{n}}}

\newcommand{\fp}{{\mathfrak{p}}}

\newcommand{\fr}{{\mathfrak{r}}}

\newcommand{\fR}{{\mathfrak{R}}}
\newcommand{\fS}{{\mathfrak{S}}}

\newcommand{\cA}{{\cmcal{A}}}
\newcommand{\cB}{{\cmcal{B}}}
\newcommand{\cC}{{\cmcal{C}}}

\newcommand{\cE}{{\cmcal{E}}}

\newcommand{\cI}{{\cmcal{I}}}
\newcommand{\cJ}{{\cmcal{J}}}

\newcommand{\cM}{{\cmcal{M}}}
\newcommand{\cN}{{\cmcal{N}}}

\newcommand{\cS}{{\cmcal{S}}}
\newcommand{\cT}{{\cmcal{T}}}

\newcommand{\cX}{{\cmcal{X}}}
\newcommand{\cY}{{\cmcal{Y}}}
\newcommand{\cZ}{{\cmcal{Z}}}

\def\a{\alpha}
\def\b{\beta}

\def\d{\delta}
\def\e{\epsilon}

\def\g{\gamma}
\def\k{\kappa}

\def\o{\omega}
\def\s{\sigma}
\def\u{\upsilon}

\def\p{\varphi}
\def\vp{\varpi}

\def\G{\Gamma}
\def\D{\Delta}

\def\<{\langle}
\def\>{\rangle}

\newcommand{\GQ}{{\operatorname{Gal}(\overline{\mathbb{Q}}/{\mathbb{Q}})}}

\newcommand{\inj}{\hookrightarrow}
\newcommand{\surj}{\twoheadrightarrow}
\newcommand{\arinj}{\ar@{^(->}}
\newcommand{\arsurj}{\ar@{->>}}
\newcommand{\arsub}{\ar@{}[r]|-*[@]{\subset}}
\newcommand{\arsup}{\ar@{}[r]|-*[@]{\supset}}
\newcommand{\arcap}{\ar@{}[d]|-*[@]{\subset}}
\newcommand{\arcup}{\ar@{}[u]|-*[@]{\subset}}

\renewcommand{\pmod}[1]{{~~(\operatorname{mod}~{#1})}}
\renewcommand{\mod}[1]{{~~\operatorname{mod}~{#1}}}

\newcommand{\Qbar}{\overline{\mathbb{Q}}}

\newcommand{\Res}{{\operatorname{Res}}}

\newcommand{\Gal}{{\operatorname{Gal}}}
\newcommand{\Div}{{\operatorname{Div}}}
\newcommand{\Pic}{{\operatorname{Pic}}}
\newcommand{\End}{{\operatorname{End}}}

\newcommand{\GL}{{\operatorname{GL}}}

\newcommand{\Frob}{{\text{Frob}}}

\newcommand{\old}{{\mathrm{old}}}
\newcommand{\new}{{\mathrm{new}}}

\newcommand{\tr}{{\operatorname{tr}}}

\newcommand{\modl}{{\pmod {\ell}}} 
\mathchardef\hyp="2D

\newcommand{\sqf}{{\mathrm{sf}}}
\newcommand{\sqq}{{\mathrm{sq}}}
\newcommand{\rrad}{{\mathrm{rad}}}

\newcommand{\mat}[4]{
 \left(  \begin{smallmatrix} #1 & #2 \\ #3 & #4 \end{smallmatrix} \right)}
 
\newcommand{\bmat}[4]{
  \left( \begin{array}{cc} #1 & #2 \\ #3 & #4 \end{array} \right)}

\newcommand{\vect}[2]{
 \left(  \begin{smallmatrix} #1 \\ #2 \end{smallmatrix} \right)}

\newcommand{\br}[1]{\langle #1 \rangle}

\newcommand{\ms}{\hspace*{1cm}}

\newcommand{\zmod}[1]{{\Z/{#1}\Z}}

\newcommand{\exclude}[1]{}

\newcommand{\qa}{\quad\text{and}\quad}
\renewcommand{~}{\hspace{0.5mm}}

\newcommand{\bd}[1]{\boldsymbol{\mathrm{#1}}}
\renewcommand{\ms}{\vspace{2mm}}

\newcommand{\red}[1]{{\color{black} #1}}

\begin{document}                                                                          

\title[On rational Eisenstein primes and the rational cuspidal groups]{On rational Eisenstein primes and the rational cuspidal groups of modular Jacobian varieties}
\author{Hwajong Yoo}
\address{Center for Geometry and Physics, Institute for Basic Science (IBS), Pohang, Republic of Korea 37673}
\email{hwajong@gmail.com}
\subjclass[2010]{11F33, 11F80, 11G18}
\keywords{Cuspidal group, Eisenstein ideals}

\begin{abstract}
Let $N$ be a non-squarefree positive integer and let $\ell$ be an odd prime such that $\ell^2$ does not divide $N$. 
Consider the Hecke ring $\mathbb{T}(N)$ of weight $2$ for $\Gamma_0(N)$, and its rational Eisenstein primes of $\mathbb{T}(N)$ containing $\ell$, defined in Section \ref{sec: rational Eisenstein primes}. If $\mathfrak{m}$ is such a rational Eisenstein prime, then we prove that 
$\mathfrak{m}$ is of the form $(\ell, ~\cmcal{I}^D_{M, N})$, where the ideal $\cmcal{I}^D_{M, N}$ of $\mathbb{T}(N)$ is also defined in Section \ref{sec: rational Eisenstein primes}. 
Furthermore, we prove that $\cmcal{C}(N)[\mathfrak{m}] \neq 0$, where $\cmcal{C}(N)$ is the rational cuspidal group of $J_0(N)$. 
To do this, we compute the precise order of the cuspidal divisor $\cmcal{C}^D_{M, N}$, defined in Section \ref{sec: the cuspidal divisor}, and the index of $\cmcal{I}^D_{M, N}$ in $\mathbb{T}(N)\otimes \Z_\ell$.
\end{abstract}
\maketitle
\setcounter{tocdepth}{1} 
\tableofcontents

\section{Introduction}
Let $N$ be a positive integer. Consider the modular curve $X_0(N)$ of level $N$ over $\Q$. This curve is endowed with a Hecke correspondence $T_p$ for each prime number $p$. 
The correspondence $T_p$ induces an endomorphism of the Jacobian variety $J_0(N):=\Pic^0(X_0(N))$ of $X_0(N)$, which is again denoted by $T_p$. Let $\T(N)$ be the $\Z$-subalgebra of $\End(J_0(N))$ generated by the family of endomorphisms $T_n$ for all integers $n\geq 1$.

Suppose that $\fm$ is a maximal ideal of $\T(N)$, and let $\ell$ be the characteristic of $\T(N)/\fm$. Let
$$
\rho_{\fm} : \GQ \rightarrow \GL_2(\T(N)/\fm)
$$
be the two-dimensional semisimple representation attached to $\fm$ (cf. \cite[\textsection 5]{R90}). Up to isomorphism, this representation is characterized by the following:
\begin{itemize}
\item it is unramified outside $\ell N$;
\item for each prime $p$ not dividing $N$, the characteristic polynomial of $\rho_{\fm}(\Frob_p)$ is
$$
X^2-T_p \pmod \fm X+ p,
$$ 
where $\Frob_p$ is a Frobenius element for a prime $p$ in $\GQ$.
\end{itemize}
If $\rho_{\fm}$ is reducible, $\fm$ is called an \textsf{Eisenstein prime} or \textsf{Eisenstein}. 

\ms
The study of Eisenstein primes goes back to the seminal paper \cite{M77} by Mazur.
For prime $N$, he introduced the \textsf{Eisenstein ideal}
$$
\cI_0(N):=(T_p-p-1 : \text{ for primes $p$ not dividing $N$}) \subseteq \T(N).
$$
This ideal then contains $T_N-1$ as well, and it annihilates 
all ``Eisenstein modules'', for example
\begin{itemize}
\item
the \textsf{cuspidal group} $\cC_N$, the group generated by the equivalence classes of degree $0$ divisors supported only on cusps;
\item
the \textsf{rational torsion subgroup} $\cT(N)$, the group of the elements of finite order in $J_0(N)(\Q)$;
\item
the \textsf{component group} of the special fiber of the N\'eron model $\cJ_0(N)$ of $J_0(N)$ over $\F_N$;
\item
the \textsf{Shimura subgroup} $\Sigma_N$, the kernel of the natural map $J_0(N) \rightarrow J_1(N)$, obtained from the natural covering $X_1(N) \rightarrow X_0(N)$.
\end{itemize} 
In fact, the quotient $\T(N)/{\cI_0(N)}$ is cyclic of order $n$, where $n$ is the numerator of $\frac{N-1}{12}$, and the orders of all the groups above are precisely $n$. If a prime number $\ell$ divides $n$, then the ideal $\fm=(\ell, ~\cI_0(N))$ is Eisenstein. Moreover, Mazur showed that $\dim_{\F_\ell} J_0(N)[\fm] =2$, where
\[
J_0(N)[\fm] := \{ x \in J_0(N)(\overline{\Q}) : Tx =0 \text{ for all } T \in \fm \}.
\]

\vspace{2mm}
On the other hand, the study of Eisenstein modules is rather complicated when $N$ is no longer  prime, due to the presence of old forms. Also, the orders of the groups above are different in general.
Nevertheless we can prove the following theorem if the level is squarefree.

\begin{thm}
Let $N$ be a squarefree positive integer, and let
$\fm \subseteq \T(N)$ be an Eisenstein prime containing $\ell$. Then, the following hold.
\begin{enumerate}
\item
The ideal $\fm$ always contains
$$
\cI_{M, N}:=\left(T_p-1, ~T_q-q, ~\cI_0(N) : \text{ for all primes } p \mid M \text{ and } q \mid \frac{N}{M} \right)
$$ 
where $M$ is a divisor of $N$. If we take the largest such number $M$ (as we will do henceforth), then we have $M \neq 1$. 

\item
If $M\neq 1$, then the quotient $\T(N)/{\cI_{M, N}}$ is cyclic of order $n$ up to powers of $2$, where $n$ is the numerator of 
$$
\frac{1}{24} \prod_{p \mid M} (p-1) \prod_{q \mid \frac{N}{M}} (q^2-1),
$$ 
where $p$ and $q$ are prime. In other words, we have 
$$\T(N)/{\cI_{M, N}} \otimes \Z[1/2] \simeq \zmod n \otimes \Z[1/2].$$

\item
If $\fm$ is new, then $q \equiv -1 \modl$ for all prime divisors $q$ of $N$ not dividing $M$.
\item
We have $\cC_N[\fm] \neq 0$.

\item
Let $\ell$ be a prime not dividing $6N$.
Assume that $q \equiv -1 \modl$ for all prime divisors $q$ of $N$ not dividing $M$. 
Let $s$ and $t$ be the numbers of prime divisors of $M$ and $\frac{N}{M}$, respectively.
Let $s_0$ be the number of prime divisors $p$ of $N$ such that $p \equiv 1 \modl$. If either $s_0 \neq s$ or $t=0$, then
$$
\dim J_0(N)[\fm]=1+s_0+t+\d, 
$$
where $\d=s$ (resp. $0$) if $s_0=s$ (resp. otherwise).

\item (Generalized Ogg's conjecture)
We have $\cT(N)[\ell^\infty] = \cC_N[\ell^\infty]$ for a prime $\ell \geq 5$, where $\cC_N[\ell^\infty]$ is the $\ell$-primary part of $\cC_N$. If $3$ does not divide $N$, then $\cT(N)[3^\infty] = \cC_N[3^\infty]$.
\end{enumerate}
\end{thm}

The first two assertions were proved by the author \cite[\textsection 2]{Yoo3}. The third one is well-known (Lemma \ref{lem:p^2 U_p eigenvalue}).
The fourth one was conjectured by Ribet, and was proved by the author (\textit{loc. cit.}).  
The fifth one was proved by Ribet and the author \cite{RY14}. 
The last one was proved by Ohta \cite{Oh14}. (If $N$ is the product of two distinct primes, then the result was slightly improved by the author \cite{Yoo5}.)

\vspace{2mm}
In this paper, we generalize some of the results above to the case of non-squarefree level. There are several discrepancies between two cases. First of all, if the level is non-squarefree, then there are some cuspidal divisors which are not rational and hence the sixth assertion above must be modified. Next, there are some Eisenstein primes which do not contain $\cI_0(N)$. (In fact, if $\ell$ does not divide $N$, $\rho_{\fm}$ is isomorphic to $\alpha \oplus \chi_{\ell} \cdot \alpha^{-1}$, where $\alpha$ is a character of $\GQ$ unramified outside $N$ and $\chi_{\ell}$ is the mod $\ell$ cyclotomic character, by Faltings and Jordan \cite[Th. 2.3]{FJ95}. The possible characters $\alpha$ can be studied by a method of Mazur \cite[\textsection 5]{M78}, and we can also construct such non-trivial characters explictly. In this direction, see also the paper of Stevens \cite{Ste85}.) Finally, if $\ell^2$ divides $N$, then the geometry of the special fiber of the N\'eron model of $J_0(N)$ at $\ell$ is rather difficult.

One possible generalization of the statements above to the case of  non-squarefree level is to impose the condition of ``rationality''. More precisely, we say
that an Eisenstein prime $\fm$ of $\T(N)$ is \textsf{rational Eisenstein} if $\rho_{\fm} \simeq \mathbbm{1} \oplus \chi_{\ell}$, where $\mathbbm{1}$ is the trivial character, or equivalently it contains $\cI_0(N)$. 
The intersection $\cC(N):=\cC_N \cap \cT(N)$ is called the \textsf{rational cuspidal group of $J_0(N)$}. It can also be defined as the subgroup generated by equivalence classes of degree $0$ cuspidal divisors stable under the action of $\GQ$. (By a theorem of Manin and Drinfeld \cite{Dr73, Ma72}, the points on $J_0(N)$ defined by such classes are of finite order, and hence they are contained in $\cT(N)$.)
Let us remark that if $\cT(N)[\fm] \neq 0$, then $\fm$ is rational Eisenstein by the Eichler--Shimura relation and a standard argument using the Brauer--Nesbitt Theorem and Chebotarev density theorem (cf. \cite[Ch. II, \textsection 14]{M77}).

The following conjecture is an optimistic generalization of the above theorem to non-squarefree $N$. See Section \ref{sec : notation} for unfamiliar notation.
\begin{conj}\label{conjecture}
Let $N$ be a non-squarefree positive integer and let $\fm \subseteq \T(N)$ be a rational Eisenstein prime containing $\ell$. Then, the following hold.
\begin{enumerate}
\item
The ideal $\fm$ always contains the ideal
$$
\cI_{M, N}^D:=\left (T_p-1, ~T_q-q,~T_\fp, ~\cI_0(N) : \text{ for primes } p \mid M, ~q \mid \frac{N^{\rrad}}{MD}, ~\fp \mid D \right )
$$
where $M$ is a divisor of $N^{\rrad}$ and $D$ is a divisor of $N^{\sqq}$ prime to $M$. If we take the largest such numbers $M$ and $D$ (as we will do henceforth), then we have $MD \neq 1$.

\item
If $MD\neq 1$, then
the index of $\cI_{M, N}^D$ is equal to the order of some cuspidal divisor up to powers of $2$.

\item
If $\fm$ is new, then $D=N^{\sqq}$, and $q\equiv -1 \modl$ for all prime divisors $q$ of $N^\sqf$ not dividing $M$.
\item
We have $\cC(N)[\fm] \neq 0$.

\item
Let $\ell$ be a prime not dividing $6N$.
Assume that $D=N^{\sqq}$ and $q \equiv -1 \modl$ for all prime divisors $q$ of $N^\sqf$ not dividing $M$.
Let $s_0$ and $u_0$ be the numbers of prime divisors of $M$ and $N^{\sqq}$ which are  congruent to $1$ modulo $\ell$, respectively. Also, let $u_1-u_0$ be the number of prime divisors of $N^{\sqq}$ which are congruent to $-1$ modulo $\ell$. Then, we have
\begin{equation*}
\dim J_0(N)[\fm]=1+s_0+t+u_1,
\end{equation*}
where $t$ is the number of prime divisors of $N^\sqf$ not dividing $M$.

\item
(Generalized Ogg's conjecture) We have $\cT(N)=\cC(N)$.
\end{enumerate}
\end{conj}

The fifth assertion was studied by the author in \cite{Yoo7}; in particular, it was proved
if either $s_0 \neq s$ or $t=0$ or $u_0 \neq u$, where $s$ and $u$ are the numbers of the prime divisors of $M$ and $N^{\sqq}$, respectively. The last assertion was considered by Lorenzini in \cite{Lo95} when the level is a power of a prime greater than $3$.
In this paper, we will discuss the remaining assertions. More precisely, we have the following.
\begin{thm} \label{thm : main thm}
Let $\fm \subseteq \T(N)$ be a rational Eisenstein prime containing $\ell$. Suppose that $\ell$ is odd and $\ell^2$ does not divide $N$. Then, the following hold.
\begin{enumerate}
\item
The ideal $\fm$ is of the form $(\ell,~\cI_{M, N}^D)$ for some divisor $M$ of $N^{\rrad}$ and some divisor $D$ of $N^{\sqq}$ prime to $M$. If we take the largest such numbers $M$ and $D$ (as we will do henceforth), then we have $MD\neq 1$.

\item
If $MD \neq 1$, then we have
$$\T(N)/{\cI_{M, N}^D} \otimes \Z_\ell \simeq \zmod n \otimes \Z_\ell, $$
where $n$ is the order of the cuspidal divisor $\cC_{M, N}^D$.

\item
If $\fm$ is new, then $D=N^{\sqq}$ and $q \equiv -1 \modl$ for all prime divisors $q$ of $N^\sqf$ not dividing $M$.

\item
We have $\cC(N)[\fm] \neq 0$.

\end{enumerate}
For $\ell=2$, if $N$ is not divisible by $4$ and $N^{\rrad}/M$ is an odd integer greater than $1$, then the same results also hold true.
\end{thm}

\noindent (For the definition of $\cC_{M, N}^D$, see Section \ref{sec: the cuspidal divisor}.) 

\red{\begin{rem}
Let $\ell$ be a prime not dividing $6N^{\sqq}$ and let $\fm=(\ell, \cI_{M, N}^D)$. 
Assume that $D=N^{\sqq}$ and $q\equiv -1 \pmod \ell$ for all prime divisors $q$ of $N^\sqf$ not dividing $M$ (as in the third statement).
Then, by the assertions (1) and (2) above (and the computation of $n$ in Theorem \ref{thm : order of C_M, N}) we have
\begin{equation*}
\text{$\fm$ is maximal} \iff s+u\geq 1 \text{ and } s_0+t+u_1 \geq 1
\end{equation*}
where $s_0, s, t, u_0, u_1$ and $u$ are defined in Conjecture \ref{conjecture}(5).
\end{rem}}

\begin{proof}[Proof of Theorem \ref{thm : main thm}]
The first and third assertions are proved in Theorem \ref{thm : classification of Eisenstein ideals}. The second one is proved in Theorem \ref{thm : the index}. The last one easily follows from the first and second ones because $\cC_{M, N}^D$ is annihilated by $\cI_{M, N}^D$ (Theorem \ref{thm : hecke action on C_M, N}).
\end{proof}
The idea of the proof of the second assertion (Theorem \ref{thm : the index}) is as follows. Let $\fm=(\ell, ~\cI_{M, N}^D)$ be a rational Eisenstein prime.
First, we study a rational cuspidal divisor $\cC_{M, N}^D$, which is annihilated by $\cI_{M, N}^D$. More precisely, we compute the precise order of $\cC_{M, N}^D \in J_0(N)(\Q)$ (Theorem \ref{thm : order of C_M, N} and Corollary \ref{cor : the order of C_M, N}) and the action of the Hecke operators on it (Theorem \ref{thm : hecke action on C_M, N}). This gives a lower bound of the index of $\cI_{M, N}^D$.  
Then, we define an Eisenstein series $\cE_{M, N}^D$ and compute its residues at various cusps (Lemma \ref{lem : 5.8 new}). This gives an upper bound of the index of $\cI_{M, N}^D$. Fortunately, these two bounds are ``$\ell$-adically'' equal and imply the result.

\ms
Before giving an application, we discuss a useful heuristic that works very well in the study of Eisenstein modules. Let $\a_p(N)$ and $\b_p(N)$ be two degeneracy maps from $X_0(Np)$ to $X_0(N)$ defined by ``forgetting the level $p$ structure'' and ``dividing by the level $p$ structure'', respectively. (See Section \ref{sec : degeneracy maps} for the precise definitions.)
Let $\alpha_p(N)^*, \beta_p(N)^* : J_0(N) \to J_0(Np)$ be the pullbacks of $\a_p(N), \b_p(N)$ via Picard functoriality.
Also, let $[N]^{\bigstar}_p$ be the maps from $J_0(N)$ to $J_0(Np)$ defined by
$$
\begin{cases}
[N]^+_p(x):= \alpha_p(N)^*(x)-\beta_p(N)^*(x),\\
[N]^-_p(x):=p \cdot \alpha_p(N)^*(x)-\beta_p(N)^*(x),\\
[N]_p(x):=\alpha_p(N)^*(x).
\end{cases}
$$
Similarly, we define the maps $[N]_p^\bigstar$ from the space of modular forms of weight $2$ and level $N$ to the space of those of level $Np$. 
Suppose that $N=1$. Then, there is only one cusp $i\infty$, and there is a ``fake'' modular form  
$\cE_0$ whose $q$-expansion at $i\infty$ is
\[
1- 24\sum_{n \geq 1} \sigma(n) \cdot q^n,
\]
where $\sigma(n)$ is the sum of all (positive) divisors of $n$ (cf. \cite[p. 78]{M77}). Note that the residue of $\cE_0$ at $i\infty$ is $1$, which is clear from its $q$-expansion.
Then, for a prime $N$, $[1]_N^+(\cE_0)$ is an Eisenstein series (but $[1]_N^-(\cE_0)$ is not). Similarly, (formally)\footnote{Indeed, $[1]_N^+(i\infty)=(N-1)\cdot \cC_{N, N}$, which is of degree $0$. However $[1]_N^+(i\infty)$ is zero in $J_0(N)$ because $(N-1)$ is a multiple of the order of $\cC_{N, N}$. So, to find the ``correct'' cuspidal divisor $\cC_{N, N}$ via an inductive method, we have to divide $[1]_N^+(i\infty)$ by the greatest common divisor of the coefficients.} $[1]_N^+(i\infty)$ is a divisor of degree $0$ (but $[1]_N^-(i\infty)$ is not). 
For the case of level $N^2$, we can obtain (formal) Eisenstein series $[N]^\star_N \circ [1]_N^\star(\cE_0)$, and (formal) cuspidal divisors $[N]^\star_N \circ [1]_N^\star (i\infty)$. By a similar argument as in the case of  level $N$, some are genuine Eisenstein series and some are divisors of degree $0$; and others are not\footnote{If an Eisenstein series under consideration is indeed a modular form then the sum of the residues must be $0$.}. This inductive method can produce all interesting Eisenstein series and cuspidal divisors of arbitrary level, which are stable under the action of the Hecke operators. (This is not just a heuristic, because we can start this process from prime level, which we will do in  Sections \ref{sec: the cuspidal divisor} and \ref{sec : Eisenstein series}.) The merit of this method is clear. We can obtain various information about the Eisenstein modules (e.g. the orders of cuspidal divisors and the residues of Eisenstein series) by this kind of inductive argument\footnote{The computation of the orders of cuspidal divisors in this paper is easier than the previous method of \cite[\textsection 3]{Yoo3}.
More specifically, the most difficult part of the computation is to find the inverse of the matrix $\Lambda(N)$. In \cite[\textsection 3]{Yoo3}, we introduced auxiliary notions to do it. Here, we instead compute $\Lambda(N)^{-1}$ by this inductive method, which also clarifies the previous computation.}. Furthermore, if we use an ``appropriate'' map $[N]_p^\bigstar$ in each inductive step, we can easily prove that they are annihilated by a certain Eisenstein ideal (see Proposition \ref{prop : the maps [p]} and Remark \ref{rem : eigenvalue}).

\ms
As an application, we prove the following.
\begin{thm}\label{thm:corollary}
Let $M$ be a divisor of $N^\sqf$ such that $M\cdot N^\sqq \neq 1$. Let $p$ be a prime divisor of $N$, and let $M'={M}/{\gcd(M, p)}$.
Let $[N]_p^\bullet$ be the map from $J_0(N)$ to $J_0(Np)$, and let $\cN$ be the number defined as follows.
\begin{equation*}
([N]_p^\bullet, \cN):=\begin{cases}
([N]_p^-, p+1)& \text{ if } p \mid M,\\
([N]_p^+, 1)& \text{ if } p \mid \frac{N^\sqf}{M},\\
([N]_p, p)& \text{ if } p^2 \mid N.\\
\end{cases}
\end{equation*}
Then, we have $[N]_p^\bullet(\cC_{M, N})=\cN \cdot \cC_{M', Np}$. If we let $K$ denote the intersection of the kernel of $[N]_p^\bullet$ and $\< \cC_{M, N} \>$, then we have $K \simeq \< \cC_{M, N} \>[h]$, where $h$ is either $1$ or $2$. Furthermore, $h=2$ if and only if $(M, N)=(q, q)$ or $(q, 2q)$ for $q$ an odd prime $q$. In other words, we have
\begin{equation*}
 K=\begin{cases}
 \< \cC_{M, N} \>[2] & \text{ if  either } (M, N)=(q, q) \text{ or } (q, 2q) \text{ for an odd prime }q,\\
 \quad 0  & \text{ otherwise}.
 \end{cases}
 \end{equation*}
Thus, for an odd prime $\ell$, we have
\begin{equation*}
[N]_p^\bullet (\br {\cC_{M, N}}[\ell])=\br {\cC_{M', Np}}[\ell].
\end{equation*}

\end{thm}
\noindent (Here, $\cC_{M, N}$ and $\cC_{M', Np}$ denote $\cC_{M, N}^{N^{\sqq}}$ and $\cC_{M', Np}^{(Np)^\sqq}$, respectively. For the proof, see Section \ref{sec : the map [N] on the cuspidal divisors}.) 

If $p$ does not divide $N$, then this is an easy consequence of a result proved by Ribet \cite{R83}. Indeed, when $p$ does not divide $N$ he proved that the kernel of the map 
$$
\gamma_p(Np)^* : J_0(N) \times J_0(N) \rightarrow J_0(Np)
$$
is the antidiagonal image of the Shimura subgroup of $J_0(N)$. (Also, the intersection of the cuspidal group and the Shimura group is at most of order $2$, cf. \cite[Ch. II, Prop. 11.11]{M77}.)
On the other hand, if $p$ divides $N$, then not much is known about the kernel of $\gamma_p(Np)^*$. (Some cases were done by Ling \cite{Lg95}.) The theorem above says that at least under an (appropriate) map $[N]^{\bigstar}_p$, the subgroup $\br {\cC_{M, N}}[\ell] \subseteq J_0(N)[\fm]$ maps injectively into $J_0(Np)[\fm']$ (where $\fm'$ is compatible with $\fm$). This result is used crucially in \cite{Yoo7} in order to compute the dimension of
$$
J_0(N)[\fm] := \{ x \in J_0(N)(\Qbar) : Tx = 0 \quad\text{for all }~ T \in \fm \}
$$
for non-squarefree $N$.

\subsection*{Acknowledgements} The author would like to thank Ken Ribet for very helpful discussions. He would also like to thank Jinhyun Park for helpful comments. 
He is grateful to the referee for supplying many helpful
comments and suggestions that helped him to improve some results.
This work was supported by IBS-R003-D1.

\ms
\subsection{Notation}\label{sec : notation}
Throughout this paper, $p$ and $\ell$ are always prime numbers. 
Also, $N$ is always a positive integer. By a divisor of $N$ we always mean a {\it positive} divisor of $N$, and we denote by $\cS(N)$ the set of the (positive) divisors of $N$. We denote by $\o(N)$ the number of prime divisors of $N$. We denote by $\p$ Euler's totient function, i.e., 
\begin{equation*}
\p(N):=N\cdot \prod_{p \mid N} (1-\frac{1}{p}).
\end{equation*}
For a positive integer $d$, we denote by $\nabla(d)$ the set of positive integers prime to $d$, i.e.,
\begin{equation*}
\nabla(d):=\{ x \in \Z : x \geq 1 \text{ and } \gcd(x, d)=1 \}.
\end{equation*}

We denote by $p^r \mid \mid N$ if $p^r$ divides $N$ but $p^{r+1}$ does not. In such a case, we say that \textsf{$p^r$ exactly divides $N$}, and denote by 
$\mathrm{val}_p(N):=r$ the (normalized) valuation of $N$ with respect to $p$.
We also denote by
$$
N^\sqf := \prod_{\mathrm{val}_p(N)=1} p, \quad N^\sqq := \prod_{\mathrm{val}_p(N) \geq 2} p \quad\text{and}\quad N^{\rrad}:=\prod_{p \mid N} p = N^\sqf \cdot N^\sqq.
$$

For a positive integer $N$, we define two sets $\fS(N)$ and $\fS_0(N)$ of pairs
$(M, D)$ consisting of certain divisors of $N$ as follows.
\begin{equation}\label{eqn : the set fS(N) }
\begin{split}
\fS(N)&:=\{(M, D) : \text{$M$ is a divisor of $N^{\rrad}$ and $D$ is a divisor of $N^{{\sqq}}$ prime to $M$} \}, \\
\fS_0(N)&:=\{ (M, D) \in \fS(N) : MD\neq 1 \}.
\end{split}
\end{equation}
For $(M, D) \in \fS_0(N)$, we set the number $h(M, N, D)=1$ unless one of the following holds, in which case we set $h(M, N, D)=2$:
\begin{enumerate}[(i)]
\item
$N=M$ is a prime (and so $D=1$).
\item
$N=2^k M$ with $M$ an odd prime for $k\geq 1$, and $D=1$.
\item
$M=1$, $N=2^k$ for $k\geq 2$ and $D=2$.
\end{enumerate} 
For $(M, D) \in \fS(N)$, we define
\begin{equation*}
a_p(M, N, D):=\begin{cases}
p+1 & \text{ if $p$ does not divide $N$}, \\
1 & \text{ if $p$ divides $M$}, \\
0 & \text{ if $p$ divides $D$}, \\
p & \text{ otherwise}. 
\end{cases}
\end{equation*}

For an integer $r\geq 0$, we set 
\begin{equation*}
\cX_p(r):=p^{r-1}(p-1), ~~\cY_p(r):=p^{r-1}(p^2-1), ~~\text{ and } \cZ_p(r):=p^{r-2}(p^2-1).
\end{equation*}
Then, for $(M, D) \in \fS(N)$, we define the number
\begin{equation}
\begin{split}
\cN(M, N, D):&=\prod_{p \mid M} \cX_p(e(p)) \cdot \prod_{p \mid \frac{N^{\rrad}}{MD}} \cY_p(e(p)) \cdot 
\prod_{p \mid D} \cZ_p(e(p)) \\
&=\frac{N}{N^{\rrad} \cdot D} \cdot \prod_{p \mid M} (p-1) \cdot \prod_{p \mid \frac{N^{\rrad}}{M}} (p^2-1) ,
\end{split}
\end{equation}
where $e(p)=\mathrm{val}_p(N)$.

\ms
\section{Background}\label{sec : Background}
In this section, we review well-known results on cusps of $X_0(N)$ and the action of the Hecke operators on $J_0(N)$. 
Throughout this section, we fix the notation, $r:=\mathrm{val}_p(N)$.

\subsection{The degeneracy maps} \label{sec : degeneracy maps}
Let $\alpha_p(N) : X_0(Np) \rightarrow X_0(N)$ be the degeneracy covering with the modular interpretation $(E, C) \mapsto (E, C[N])$, where $C$ is a cyclic subgroup of order $Np$ in an elliptic curve $E$. Let $\beta_{p}(N)$ be the ``other'' degeneracy covering $X_0(Np) \rightarrow X_0(N)$; it has the modular interpretation $(E, C) \mapsto (E/{C[p]}, C/C[p])$. 
If we identify the set of complex points on $X_0(Np)$ with $\Gamma_0(Np) \backslash (\fh \cup \bP^1(\Q))$, where $\fh$ is the complex upper half plane, 
then for $\tau \in \fh \cup \bP^1(\Q)$ we have
\begin{equation}\label{eqn : alpha beta on h}
\begin{cases}
\alpha_p(N)(\tau) \equiv \,\,\tau ~\pmod {\Gamma_0(N)},\\
\beta_p(N)(\tau) \equiv \,p\tau \pmod {\Gamma_0(N)}
\end{cases}
\end{equation}
(cf. \cite[\textsection 12.1]{RS97}).
The degeneracy coverings $\alpha_{p}(N)$ and $\beta_{p}(N)$ have degree $p$ if $r\geq 1$, and degree $p+1$ if $r=0$. They induce the maps 
$$
\alpha_{p}(N)_*, \beta_{p}(N)_* : J_0(Np) \rightrightarrows J_0(N), 
\quad\quad \alpha_{p}(N)^*, \beta_{p}(N)^* : J_0(N) \rightrightarrows J_0(Np)
$$
via the two functorialities of the Jacobian.

\ms
\subsection{The cusps of $X_0(N)$}\label{sec : the cusp}
The elements in $X_0(N) \setminus Y_0(N)$ are called the {\sf cusps} of $X_0(N)$. They can be regarded as the equivalence classes in $\bP^1(\Q)$ modulo $\Gamma_0(N)$. Following Ogg \cite{Og73}, for $a \in \nabla(b)$ we denote by $\vect a b$ the cusp $a/b \in \G_0(N) \backslash \bP^1(\Q)$. 
(So, we say that $\vect a b$ \textsf{is} a cusp of $X_0(N)$.)
We write $\vect a b = \vect c d$ if and only if
they are equal as cusps of $X_0(N)$, i.e., $\G_0(N)(a/b)=\G_0(N)(c/d)$.
We often allow the case $\gcd(a, b)=d\neq 1$ below for convenience and we will identify $\vect a b$ with $\vect {a/d} {b/d}$.
By \cite[Prop. 3.8.3]{DS05}, we have
\begin{equation}\label{eqn : 2.2  equivalence}
\vect a b = \vect c d \iff \vect {ya}{b} \equiv \vect {c+jd}{yd} \pmod {N},
\end{equation}
where $j$ can be any integer and $y$ is any integer prime to $N$. Therefore a cusp $\vect a b$ of $X_0(N)$ with $a \in \nabla(b)$ is equal to $\vect x d$ for some divisor $d$ of $N$ and $x \in \nabla(d)$. (Indeed, we can take $d=\gcd(b, N)$.)

\ms
Let $\d$ be a divisor of $N$. For an integer $x \in \nabla(\d)$, a cusp $\vect x \d$ of $X_0(N)$ is called a {\sf cusp of level $\d$}. 
By Equation (\ref{eqn : 2.2  equivalence}), for $x, y \in \nabla(\d)$, we have 
\begin{equation}\label{eqn : 2.3 equivalence 2}
\vect x d = \vect y d \iff  x \equiv y \pmod {\gcd(\d, N/\d)}
\end{equation}
(cf. \cite[\textsection 1]{Og73}).
Therefore the number of all cusps of level $\d$ is $\p(\gcd(\d, N/\d))$. Since any cusp of $X_0(N)$ is equal to a cusp of level $\d$ for some divisor $\d$ of $N$, the number of all cusps is $\sum_{\d \mid N} \p(\gcd(\d, N/\d))$.
We define the divisor $(P_\d)$ as the sum of all cusps of level $\d$.
Note that the degree of $(P_\d)$ is $\varphi(\gcd(\d, N/\d))$ and $(P_\d)$ is invariant under $\GQ$ (cf. \cite[\textsection 2]{Lg97}).

\ms
By Equation (\ref{eqn : alpha beta on h}), we obtain the following (cf. \cite[p. 42]{Lg97}). 
Let $d$ be a divisor of $N/{p^r}$ and let $t = \gcd(d, N/d)=\gcd(d, N/{dp^r})$, which are both prime to $p$.
Then, for an integer $x \in \nabla(dp^i)$, we have
\begin{equation}\label{equation : alpha cusp}
\alpha_p(N)\vect {x}{dp^i} =\vect {x}{dp^i} \pmod {\G_0(N)} = \begin{cases}
\vect{x}{dp^i}  & \text{ if } i \leq r, \\
\vect{px+d}{dp^r} & \text{ if } i=r+1.
\end{cases}
\end{equation} 
Indeed, since $\gcd(N/{p^r}, p)=1$, there is an integer $k$ such that $\gcd(kN/{p^r}+p, N)=1$. By Equation (\ref{eqn : 2.2  equivalence}), we have
\begin{equation*}
\vect{x}{dp^{r+1}} = \vect{x}{dp^{r+1}+kdN} =\vect{x}{dp^r(p+kN/{p^r})} = \vect{(p+kN/{p^r})x}{dp^r} = \vect{px+d}{dp^r}.
 \end{equation*}
 The last equivalence follows from Equation (\ref{eqn : 2.3 equivalence 2}) since $ px+k(N/p^r)x \equiv px \equiv px+d \pmod t$ and $\gcd(px+d, dp^r)=1$ for all $r\geq 0$.
 
Similarly, we have 
\begin{equation}\label{equation : beta cusp}
\beta_p(N) \vect{x}{dp^i} =\vect{px}{dp^i} \pmod {\G_0(N)}=\begin{cases}
\vect{x}{dp^{i-1}}  & \text{ if } i \geq 1, \\
~~ \vect{px}{d} & \text{ if } i=0. \end{cases}
\end{equation}

\ms
By the description above, for $x, y \in \nabla(dp^i)$ we have
$$
\alpha_p(N) \vect{x}{dp^i} = \alpha_p(N) \vect{y}{dp^i} 
$$
if and only if $x \equiv y \pmod {p^b t}$, where $b=\min \{i, r-i\}$ for $0\leq i \leq r$ and $b=0$ for $i=r+1$. 
Moreover, 
$$
\alpha_p(N)\vect{x}{dp^{r+1}}  = \alpha_p(N) \vect{y}{dp^r} 
$$
if and only if $px \equiv y \pmod {t}$. 
Analogously, we obtain
$$
\b_p(N) \vect{x}{dp^i} = \b_p(N) \vect{y}{dp^i}
$$
if and only if $x \equiv y \pmod {p^c t}$, where $c=\min \{i-1, r+1-i \}$ for $1\leq i \leq r+1$ and $c=0$ for $i=0$. Moreover, we have
$$
\b_p(N) \vect{x}{d} = \b_p(N) \vect{y}{dp}
$$
if and only if $px \equiv y \pmod {t}$. In summary, we have the following, which is well-known to experts.
\begin{lem} \label{lem : 2.1 degeneracy maps on cusps}
Let $d$ be a divisor of $N/{p^r}$, where $r=\mathrm{val}_p(N) \geq 0$.
The map $\a_p(N)$ is ramified at a cusp $\vect{x}{dp^i}$ if and only if $0 \leq i \leq r/2$. Also, the map $\b_p(N)$ is ramified at a cusp $\vect{x}{dp^i}$ if and only if $r/2+1 \leq i \leq r+1$.
In particular, the ramification indices of the maps $\a_p(N)$ and $\b_p(N)$ at a cusp $\vect{x}{dp^i}$ depend only on $i$, neither on $x \in \nabla(dp^i)$ nor on $d$.
Furthermore, the ramification indices of $\a_p(N)$ and $\b_p(N)$ at a cusp are either $1$ or $p$. 
\end{lem}
\begin{proof}
Using the formula (4.5) on \cite[p. 538]{Ste85}, we can prove the statement for the map $\a_p(N)$ 
by comparing the ramification indices of $\vect {x}{dp^i}$ and $\a_p(N)\vect {x}{dp^i}$ over a cusp $\vect 1 1 \in X_0(1)$. For instance, if $r=0$ then the map $\a_p(N)$ is ramified at a cusp $\vect x d$ (of ramification index $p$) and unramified at a cusp $\vect x {dp}$ for any $x \in \nabla(dp)$. 

Now, consider the map $\b_p(N)$. If $r=0$, then we use the Atkin-Lehner theory \cite{AL70}. Let $w_p$ be the Atkin-Lehner involution on $X_0(Np)$ with respect to $p$. As a matrix form, we can take $W_p=\mat {ap} {-b} {N^2 p} p$, where $a$ and  $b$ are taken so that $ap+bN^2=1$. (This is possible because $\gcd(N^2, p)=1$.) Then, we have
\begin{equation*}
w_p  \vect{x}{d} =W_p \vect{x}{d} =\vect{apx-bd}{Np(Nx)+dp}  =\vect{y}{dp},
\end{equation*}
where $y=x$ if $\gcd(x, p)=1$ and $y=x+d$ otherwise. (Note that $\gcd (y, dp)=1$.) The last equivalence follows because $apx-bd \equiv x \equiv y \pmod {d}$ and of course $d$ is a multiple of $\gcd(d, Np/d)=\gcd(d, N/d)$.
Similarly, we have
\begin{equation*}
w_p \vect{x}{dp} =W_p \vect {x}{dp}=\vect{apx-bdp}{N^2px+dp^2} =\vect {ax-bd}{N^2 x+dp}=\vect{ax-bd}{d(N(N/d)x+p)}. 
\end{equation*}
Since $\gcd(N, p)=\gcd(x, dp)=1$, we have $\gcd(N(N/d)x+p, Np)=1$ and hence
\begin{equation*}
\vect {ax-bd}{d(N(N/d)x+p)}=\vect {(ax-bd)(N(N/d)x+p)}{d}=\vect {apx}{d} \equiv \vect {x}{d}.
\end{equation*}
(This can be also checked using $w_p^2=1$ on $X_0(Np)$.)
Since $\b_p(N)=\a_p(N) \circ w_p$,\footnote{On the cusps, we can also verify this equality by direct computation using the previous description.} the ramification index of $\b_p(N)$ at a cusp $\vect x d$ (resp. $\vect y {dp}$) is $1$ (resp. $p$). If $r\geq 1$, then we compute the ramification index of the map $\b_p(N)$ at a cusp $\vect {x}{dp^a}$ by searching non-equivalent cusps $\vect{y}{dp^i}$ on $X_0(Np)$ such that $\b_p(N)\vect{y}{dp^i} = \vect{x}{dp^a}$, where $i$ is either $a$ or $a+1$. This computation is done using the discussion above and the result follows for the case $r\geq 1$ because the degree of $\b_p(N)$ is $p$, which implies that the ramification index is either $1$ or $p$.
\end{proof} 

\ms
\subsection{The Hecke ring} \label{sec:hecke ring}
We review the results on the Hecke ring acting on Jacobian varieties and the space of modular forms. 
Let $\alpha_{p}(N)'$ and $\beta_{p}(N)'$ denote the transposes of $\alpha_{p}(N)$ and $\beta_{p}(N)$ viewed as correspondences, respectively. We define the Hecke correspondence $T_p$ on $X_0(N)$ by $\alpha_{p}(N)\circ \beta_{p}(N)'$:
$$
\xymatrix{
&X_0(Np) \ar[dl]_-{\beta_p(N)}\ar[dr]^-{\alpha_p(N)} & \\
X_0(N) \ar@{-->}[rr]^-{T_p}& & X_0(N).
}
$$

By direct computation (via the modular interpretation), we have
\begin{equation}
\begin{aligned}
\alpha_{p}(N) \circ \beta_{p}(N)'=\beta_{p}(N/p) ' \circ \alpha_{p}(N/p) \quad\text{ for } r \geq 2 \,;\\
\alpha_{p}(N)\circ \beta_{p}(N)'=\beta_{p}(N/p)' \circ \alpha_{p}(N/p)-w_p  \quad\text{ for } r=1,
\end{aligned}
\label{eqn : Hecke correspondence}
\end{equation}
where $w_p$ is the Atkin-Lehner involution on $X_0(Np)$ with respect to $p$ (cf. \cite[\textsection 13]{MR91}). 

\begin{defn} \label{defn : Hecke operator}
The \textsf{$p$-th Hecke operator $T_p \in \End(J_0(N))$} is the pullback of the correspondence $T_p$ to $J_0(N)$ (cf. \cite[\textsection 13]{MR91}). Namely, we define
$$
\xymatrix{
T_p := \beta_p(N)_* \circ \alpha_p(N)^* : J_0(N) \ar[r]^-{\alpha_p(N)^*} & J_0(Np) \ar[r]^-{\beta_p(N)_*} & J_0(N).
}
$$ 
The \textsf{$n$-th Hecke operator $T_n$} are inductively defined as follows.
\begin{itemize}
\item
$T_1 = 1$;

\item
$T_{ab}=T_aT_b$ if $(a, b)=1$;

\red{\item
$T_{p^k}=T_p T_{p^{k-1}}-pT_{p^{k-2}}$ for $k\geq 2$ if $p$ does not divide $N$;
\item
$T_{p^k}=T_p^k$ if $p$ divides $N$.}

\end{itemize}
The \textsf{Hecke ring $\T(N)$ of level $N$} is the subring of $\End(J_0(N))$ generated (over $\Z$) by $T_n$ for all integers $n\geq 1$. Sometimes (in Sections \ref{sec : Eisenstein series} and \ref{sec : the index}), $\T(N)$ is regarded as the subring of the endomorphism ring of the space of cusp forms of weight $2$ for $\Gamma_0(N)$ generated by the same named operators $T_n$ for all integers $n\geq 1$ (cf. \cite[\textsection 1]{MR91}).
\end{defn}

Suppose that $r\geq 1$, and let $T_p$ (resp. $\tau_p$) be $p$-th Hecke operator in $\T(N)$ (resp. $\T(N/p)$). Then by the formula (\ref{eqn : Hecke correspondence}), we have
\begin{equation}
\label{eqn : Hecke operator}
\begin{aligned}
T_p =  \alpha_p(N/p)^* \circ \beta_p(N/p)_* \quad\text{for } r \geq 2 \,;
\\
T_p +w_p = \alpha_p(N/p)^* \circ \beta_p(N/p)_* \quad\text{for } r = 1.
\end{aligned}
\end{equation}
(Here, $w_p$ is the Atkin-Lehner involution on $J_0(N)$ with respect to $p$.)
Also, we have
\begin{equation}
\label{eqn : Hecke, degeneracy}
\begin{aligned}
T_p \circ \alpha_{p}(N/p)^* = \alpha_{p}(N/p)^* \circ \tau_p,  \quad 
T_p \circ \beta_{p}(N/p)^* = p \cdot \alpha_{p}(N/p)^* \quad & \text{for } r \geq 2 \,;
\\
T_p \circ \alpha_{p}(N/p)^* = \alpha_{p}(N/p)^* \circ \tau_p - \beta_p(N/p)^*, \quad 
T_p \circ \beta_{p}(N/p)^* = p \cdot \alpha_{p}(N/p)^* \quad & \text{ for } r = 1
\end{aligned}
\end{equation}
(cf. \cite[\textsection 13]{MR91}).
\subsection{Old and new} \label{sec:old and new}
Throughout this subsection, we assume that $r\geq 1$. In other words, $p$ is a prime divisor of $N$.
We define the map
$$
\gamma_{p}(N)^* : J_0(N/p)\times J_0(N/p) \rightarrow J_0(N)
$$ 
by the formula $\gamma_{p}(N)^* (x, \,y) = \alpha_{p}(N/p)^*(x) + \beta_{p}(N/p)^*(y)$. 

Let $J:=J_0(N)$ and $\T:=\T(N)$. The image of $\gamma_{p}(N)^*$ is called the \textsf{$p$-old subvariety of $J$} and denoted by $J_{p\hyp\old}$. 
Similarly, the subvariety generated by the $p$-old subvarieties of $J$ for all prime divisors $p$ of $N$ is called the {\sf old subvariety of $J$} and denoted by $J_{\old}$. 
The quotient of $J$ by $J_{p\hyp\old}$ (resp. by $J_\old$) is called the \textsf{$p$-new quotient of $J$} (resp. {\sf new quotient of $J$}) and denoted by $J^{p\hyp\new}$ (resp. $J^\new$).

By the formula (\ref{eqn : Hecke, degeneracy}), we have the matrix relations viewed as endomorphisms of $J_0(N/p)\times J_0(N/p)$ (cf. \cite[p. 499]{Wi95})
\begin{equation} \label{eqn : matrix relation}
T_p = \bmat {\tau_p}{p}{0}{0}   \text{ for }~ r\geq 2 \qa T_p = \bmat {\tau_p}{p}{-1}{0}  \text{ for }~r=1,
\end{equation}
which make $\gamma_{p}(N)^*$ to be Hecke-equivariant, where $\tau_p$ is the $p$-th Hecke operator in $\T(N/p)$. 
(For a prime $q \neq p$, the $q$-th Hecke operator $T_q$ in $\T(N/p)$ acts diagonally on $J_0(N/p)\times J_0(N/p)$.)
The image of $\T$ in $\End(J_{p\hyp\old})$ (resp. $\End(J^{p\hyp\new})$) 
is called the \textsf{$p$-old} (resp. \textsf{$p$-new}) \textsf{quotient of $\T$} and denoted by $\T^{p\hyp\old}$ (resp. $\T^{p\hyp\new}$). Similarly, we define the {\sf new quotient $\T^\new$ of $\T$}. 
A maximal ideal of $\T$ is called \textsf{$p$-old} (resp. \textsf{$p$-new}) if its image in $\T^{p\hyp\old}$ (resp. $\T^{p\hyp\new}$) is still maximal. 
Similarly, a maximal ideal of $\T$ is called {\sf new} if its image in $\T^\new$ is still maximal.
Note that a maximal ideal of $\T$ is either $p$-old or $p$-new (or both); and a new maximal ideal is $p$-new for all prime divisors $p$ of $N$.

From the description above and the Cayley-Hamilton theorem, we have 
\begin{equation}
\label{eqn : T_p on old}
\begin{aligned}
T_p^2- \tau_p T_p = 0 \quad\text{for } r \geq 2 \, ;\\
T_p^2- \tau_p T_p + p = 0 \quad\text{for } r = 1
\end{aligned}
\end{equation}
as endomorphisms of $J_0(N/p) \times J_0(N/p)$. If we consider the Hecke ring $\T(N/p)$ and the $p$-old quotient $\T^{p\hyp\old}$ as subrings of $\End(J_0(N/p)^2)$, then they share a common subring $R$ of $\End(J_0(N/p)^2)$ generated by the $T_n$ with $n$ prime to $p$ and we have
\begin{equation*}
\T(N/p)=R[\tau_p] \qa \T^{p\hyp\old}=R[T_p].
\end{equation*}
Two Hecke operators $\tau_p$ and $T_p$ commute with each other and with the elements of $R$. Moreover, they are connected by the quadratic equation (\ref{eqn : T_p on old}). 
Let $\fa$ and $\fb$ be maximal ideals of $\T(N/p)$ and of $\T^{p\hyp\old}$, respectively. We say that they are {\sf compatible} if there is a maximal ideal of the ring $\fR=R[\tau_p, T_p]$ which contains them both (cf. \cite[\textsection 7]{R90}). 
We also say that a maximal ideal $\fm$ of $\T(N)$ is {\sf compatible} with a maximal ideal $\fn$ of $\T(N/p)$ if $\fm$ is $p$-old and the image of $\fm$ in $\T(N)^{p\hyp\old}$ is compatible with $\fn$. Furthermore, for any divisor $M$ of $N$, a maximal ideal $\fa$ of $\T(M)$ is {\sf compatible} with a maximal ideal $\fm$ of $\T(N)$ if both of the following hold.
\begin{itemize}
\item
There are divisors $N_i$ of $N$ such that 
$N_0=M$, $N_n=N$ and $N_{i+1}/{N_i}$ is prime for all $i$.
\item
There are maximal ideals $\fm_i$ of $\T(N_i)$ such that $\fm_i$ is compatible with $\fm_{i+1}$ for all $i$.
\end{itemize}

\vspace{2mm}
By the formula (\ref{eqn : Hecke operator}), the image of $T_p$ (resp. $T_p+w_p$) is contained in the $p$-old subvariety and hence $T_p$ (resp. $T_p+w_p$) maps to $0 \in \T^{p\hyp\new}$ if $r\geq 2$ (resp. $r=1$). Therefore we have the following.

\begin{lem}\label{lem:p^2 U_p eigenvalue}
Let $\fm$ be a maximal ideal of $\T(N)$. Then, the following holds.
\begin{enumerate}
\item
Assume that $\fm$ is $p$-new. Then, we have
\begin{equation*}
\begin{cases}
T_p -1 \in \fm \text{ or } ~T_p+1 \in \fm & \text{ if } r=1,\\
\quad T_p \in \fm & \text{ if } r \geq 2.
\end{cases}
\end{equation*}

\item
Assume that $r\geq 2$ and $T_p \not \in \fm$.
Then there is a maximal ideal $\fn$ of $\T(N/p^{r-1})$ compatible with $\fm$. Furthermore, we have
\[
T_p - \kappa \in \fm \Longleftrightarrow \tau_p - \kappa \in \fn,
\]
where $\tau_p$ is the $p$-th Hecke operator in $\T(N/{p^{r-1}})$.
\end{enumerate}
\end{lem}
\begin{proof}
The first claim is clear from the discussion above because $w_p$ is an involution. For the second one,  let $\e(n)$ be the image of $T_n$ in $\T(N)/\fm$. (Since $\fm$ is maximal, the quotient $\T(N)/\fm$ is a field.) 
Then, we have
\begin{equation*}
\fm=(\ell, T_q-\e(q) : \text{ for all primes } q ).
\end{equation*}
Assume that $T_p \not\in \fm$. Then, $\fm$ is not $p$-new, so it is $p$-old.
Moreover, the ideal 
\begin{equation*}
\fa = (\ell, T_p-\tau_p, T_p-\e(p), T_q-\e(q) : \text{ for all primes } q \neq p)
\end{equation*}
is a maximal ideal of $\fR$. (Note that $T_p$ and $\tau_p$ satisfy the quadratic equation $T_p^2-\tau_p T_p=0$ in $\fR$.) Thus, the ideal
\begin{equation*}
\fn = (\ell, \tau_p-\e(p), T_q-\e(q) : \text{ for all primes } q \neq p)
\end{equation*}
is also maximal because $\T(N/p)/\fn \simeq \fR/\fa \simeq \T(N)/\fm$ is a field. This proves the lemma for $r=2$. Applying this argument inductively (which is possible because $\tau_p \not \in \fn$), the result follows.
\end{proof}

\subsection{The maps $[N]^\bigstar_p$} \label{sec : themap [N]_p}
Let $\alpha_p(N)$ and $\beta_p(N)$ be two degeneracy maps from $X_0(Np)$ to $X_0(N)$ defined in Section \ref{sec : degeneracy maps}. 
\begin{defn}
For $x \in J_0(N)$, we define the maps $[N]^\bigstar_p : J_0(N) \rightarrow J_0(Np)$ by
\begin{equation}
\begin{cases}
[N]^+_p(x):=\alpha_p(N)^*(x)-\beta_p(N)^*(x) \\
[N]^-_p(x):=p \cdot \alpha_p(N)^*(x)-\beta_p(N)^*(x) \\
[N]_p(x):=\alpha_p(N)^*(x).
\end{cases}
\end{equation}
We also have the maps $[N]^\bigstar_p$ from the space of modular forms of weight $2$ for $\Gamma_0(N)$ to that for $\Gamma_0(Np)$ as follows. For a modular form $\cE$ of weight $2$ for $\Gamma_0(N)$, we define  the maps $[N]^\bigstar_p : M_2(\Gamma_0(N)) \to M_2(\Gamma_0(Np))$ by
\begin{equation}
\begin{cases}
[N]^+_p(\cE)(z):=\alpha_p(N)^*(\cE)(z)-\beta_p(N)^*(\cE)(z) = \cE(z)-p \cdot \cE(pz) \\
[N]^-_p(\cE)(z):=p\cdot \alpha_p(N)^*(\cE)(z)-\beta_p(N)^*(\cE)(z)=p\cdot (\cE(z)- \cE(pz)) \\
[N]_p(\cE)(z):=\alpha_p(N)^*(\cE)(z)=\cE(z).
\end{cases}
\end{equation}
\end{defn}
\noindent
The merit of these definitions is that they preserve ``Eisenstein modules''. In other words, we have the following.

\begin{prop} \label{prop : the maps [p]}
Let $T_p$ and $\tau_p$ denote the $p$-th Hecke operators in $\T(Np)$ and $\T(N)$, respectively.
Let $x \in J_0(N)$. Then, we have the following.
\begin{enumerate}
\item
If $r=0$ and $(\tau_p-p-1)x=0$, then $(T_p-1)[N]^+_p(x)=0$ and $(T_p-p)[N]^-_p(x)=0$.
\item
Suppose that $r \geq 1$. 
\begin{itemize}
\item
If $(\tau_p-1)x=0$, then $(T_p-0) [N]^-_p(x)=0$.
\item
If $(\tau_p-p)x=0$, then $(T_p-0) [N]^+_p(x)=0$.
\item
If $(\tau_p-\kappa)x=0$, then $(T_p-\kappa)[N]_p(x)=0$.
\end{itemize}
\end{enumerate}
If we denote by $T_p$ and $\tau_p$ the $p$-th Hecke operators acting on the spaces of modular forms of weight $2$ for $\Gamma_0(Np)$ and $\Gamma_0(N)$, respectively, then the same statement is true for a modular form $x=\cE$ of weight $2$ for $\Gamma_0(N)$.
\end{prop}
\begin{proof}
This is clear by the discussions in Section \ref{sec:hecke ring}, especially by the matrix relation (\ref{eqn : matrix relation}).
\end{proof}

\begin{rem} \label{rem : eigenvalue}
Since two degeneracy maps $\alpha_p(N)$ and $\beta_p(N)$ commute with the Hecke operators $T_n$ 
if $(n, p)=1$, so do $[N]^\bigstar_p$. Thus, $[N]^\bigstar_p$ preserve eigenspaces of all the Hecke operators $T_n$ if $(n, p)=1$.
\end{rem}

\begin{rem}\label{rem : injectivity [N]_p}
The map $[N]_p$ on $J_0(N)$ is always injective for any pair $(p, N)$. This follows from Ribet \cite[Th. 4.3]{R83} if $p$ does not divide $N$. Indeed, Ribet proved that the kernel of the map $\gamma : J_0(N) \times J_0(N) \to J_0(Np)$ is the antidiagonal embedding of the Shimura subgroup of $J_0(N)$, hence $[N]_p$ is injective. (The map $[N]_p$ is the composition of $\gamma$ and the embedding of $J_0(N)$ into the first component of $J_0(N)\times J_0(N)$.)
If $p$ divides $N$, then
$\alpha_p(N)$ is totally ramified at the cusp $0=\vect 1 1$, which implies the claim. (Note that by ``geometric class field theory'' for a given map $X \to Y$ between compact algebraic curves, the kernel of the induced map $J_Y \to J_X$ between their Jacobian varieties is isomorphic to the Cartier dual to the Galois group of the maximal unramified covering of $Y$ between $X \to Y$. Hence if $X \to Y$ is totally ramified at some point, then the induced map $J_Y \to J_X$ is always injective.)
\end{rem}

\ms
\section{Rational Eisenstein primes}\label{sec: rational Eisenstein primes}
In this section, we classify all rational Eisenstein primes. The main idea is analogous to that of \cite[\textsection 2]{Yoo3}.

\vspace{2mm}
Let $\fm$ be a rational Eisenstein prime containing a prime $\ell$, i.e., $\rho_{\fm} \simeq \mathbbm{1} \oplus \chi_{\ell}$. 
Then it contains 
$$
\cI_0(N):=(T_p-p-1:  \text{ for primes } p \nmid N)
$$ 
because 
$$
T_p \pmod {\fm} = \tr(\rho_{\fm}(\Frob_p))=1+p
$$
for a prime $p$ not dividing $\ell N$, and $T_{\ell}=1\equiv 1+\ell \pmod {\fm}$ if $\ell$ does not divide $N$ by Ribet \cite[\textsection 2]{Yoo3}.  (Although Ribet proved this fact only for $\ell \geq 3$, the same argument works for $\ell=2$.)

As before, let $\epsilon(p)$ denote the image of $T_p$ in $\T(N)/\fm$. Then, we have the following.

\begin{lem}\label{lem : e(p)=1, p, 0} 
We have $\e(p) \in \{1, p, 0\}$. More precisely, 
$$
\epsilon(p) \in 
\begin{cases}
\{1, p \}  & \text{ if } p \mid \mid N,\\
\{1, p, 0 \} & \text{ if } p^2 \mid N.
\end{cases}
$$
\end{lem}
\begin{proof}
If $p \mid \mid N$, the result follows from the same argument as in \cite[Lem. 2.1]{Yoo3}.

Let $r:=\mathrm{val}_p(N)$, and assume that $r\geq 2$. If $\fm$ is $p$-new, then $T_p \in \fm$ by Lemma \ref{lem:p^2 U_p eigenvalue} and hence $\e(p)=0$.
Suppose that $T_p \not\in \fm$. Then, again by Lemma \ref{lem:p^2 U_p eigenvalue}, there is a maximal ideal $\fn$ of $\T(N/{p^{r-1}})$ compatible with $\fm$. Since compatible maximal ideals share the common eigenvalues for the Hecke operators $T_q$ for primes $q\neq p$, $\fn$ also contains $\cI_0(N/{p^{r-1}})$. Thus, by the case above either $\tau_p -1 \in \fn$ or $\tau_p-p \in \fn$, where $\tau_p$ is the $p$-th Hecke operator in $\T(N/{p^{r-1}})$.
Therefore either $T_p-1 \in \fm$ or $T_p-p \in \fm$ by Lemma \ref{lem:p^2 U_p eigenvalue}.
\end{proof}

\begin{defn}\label{defn : cI_M, N}
Let $(M, D) \in \fS(N)$. Namely, $M$ is a divisor of $N^{\rrad}$ and 
$D$ is a divisor of $N^{\sqq}$ prime to $M$. Then, we define an ideal
\begin{equation*}
\begin{split}
\cI^D_{M, N} &:= (T_p-a_p(M, N, D) : \text{ for all primes } p)\\
& =\left (T_p-1, ~T_q-q,~T_\fp, ~\cI_0(N) : \text{ for primes } p \mid M, ~q \mid  \frac{N^{\rrad}}{MD}, ~\fp \mid D \right ).
\end{split}
\end{equation*}
If $D=N^{\sqq}$, then it is simply denoted by $\cI_{M, N}$.
Note that if $q \equiv 1 \modl$ for a prime divisor $q$ of $\frac{N^{\rrad}}{MD}$, then $(\ell, ~\cI^D_{M, N})=(\ell,~\cI^D_{Mq, N})$. Thus, when we say that \textsf{$\fm$ is of the form $(\ell, ~\cI^D_{M, N})$} or $\fm :=(\ell,~\cI^D_{M, N})$, then we \textit{always} assume that $q \not\equiv 1 \modl$ for all prime divisors $q$ of $\frac{N^{\rrad}}{MD}$. 
\end{defn}

In conclusion, we have the following theorem.
\begin{thm}\label{thm : classification of Eisenstein ideals}
Let $\fm$ be a rational Eisenstein prime of $\T(N)$ containing a prime $\ell$. 
Then, $\fm$ is of the form $(\ell, ~\cI^D_{M, N})$ for some $(M, D) \in \fS_0(N)$.
Furthermore if $\fm$ is new, then $D=N^{\sqq}$ and $q \equiv -1 \modl$ for all prime divisors $q$ of $N^\sqf$ not dividing $M$.
\end{thm}
\begin{proof}
Let $M$ be the product of the prime divisors $p$ of $N$ such that $\e(p)=1$, and let $D$ be the product of the prime divisors $q$ of $N^{\sqq}$ such that $\e(q)=0$. Then by Lemma \ref{lem : e(p)=1, p, 0}, $\fm:=(\ell, \cI_{M, N}^D)$ with $(M, D) \in \fS(N)$. 
Thus, to prove the first claim it suffices to show that $MD \neq 1$. Suppose that $M=D=1$, i.e., we assume that for any prime divisor $p$ of $N$ we have $T_p-p \in \fm$ and $p\not\equiv 1 \pmod \ell$. 

Let $p$ be a prime divisor of $N^{\sqq}$ and let $r=\mathrm{val}_p(N)\geq 2$.
Since $T_p \not\in \fm$ but $T_p- p \in \fm$, we have $p \not \in \fm$ and hence $p \neq \ell$. Furthermore since $T_p \not \in \fm$, by Lemma \ref{lem:p^2 U_p eigenvalue}, there is a maximal ideal $\fn$ of level $N/{p^{r-1}}$ compatible with $\fm$, and so $\fn$ is of the form $(\ell, \cI_{1, N/{p^{r-1}}}^1)$. Applying the same method for all the prime divisors of $N^{\sqq}$, we can find a maximal ideal $\fa:=(\ell,~\cI_{1, N^{\rrad}})$ of $\T(N^{\rrad})$ compatible with $\fm$. If $\ell$ is odd, this contradicts \cite[Prop. 5.5]{Yoo3}. 
If $\ell=2$, then $T_p-p \equiv T_p-1 \pmod \ell$ if $p$ is an odd prime. 
Thus, by our convention $N^{\sqq}=1$ and $N^{\rrad}$ must divides $2$.
However this is a contradiction because $X_0(1)$ and $X_0(2)$ are of genus 0.

For the last claim, suppose that $\fm$ is new (and hence it is $p$-new for any prime divisor $p$ of $N$). By Lemma \ref{lem:p^2 U_p eigenvalue}, for any prime divisor $p$ of $N^{\sqq}$ we have $T_p \in \fm$. Thus, we obtain $D=N^{\sqq}$.
 Now assume that $q$ divides $N^\sqf$ but it does not divide $M$. Then, we have $T_q-q \in \fm$.
Since $\fm$ is $q$-new and $q$ exactly divides $N$,  we have either $T_q-1 \in \fm$ or $T_q+1 \in \fm$ by Lemma \ref{lem:p^2 U_p eigenvalue}. Since $q \not \equiv 1 \modl$, we have $q \equiv -1 \modl$, as desired.
\end{proof}

\ms
\section{The cuspidal divisor $\cC^D_{M, N}$ on $J_0(N)$}\label{sec: the cuspidal divisor}
In this section, for a pair $(M, D) \in \fS(N)$ we define a cuspidal divisor $C_{M, N}^D$ on $X_0(N)$. For $(M, D) \in \fS_0(N)$, we compute the order of the equivalence class $\cC_{M, N}^D$ of $C_{M, N}^D$ in $J_0(N)$. 
The result of this section generalizes \cite{Lg97} and \cite{Yoo3} to the case of non-squarefree level. We also show that the cyclic group $\langle \cC^D_{M, N} \rangle$ generated by $\cC_{M, N}^D$ is annihilated by $\cI^D_{M, N}$.

Throughout this section, we frequently use the notation of Sections \ref{sec : notation} and \ref{sec : the cusp} without referring to it.

\ms
\subsection{Definition of $\cC^D_{M, N}$}
Let $(M, D) \in \fS(N)$. If $M$ is a divisor of $N^\sqf$, then we define a divisor
$$
C^D_{M, N}:=\sum\limits_{\d \mid MD}  (-1)^{\omega(\d)} \cdot \varphi\left ( \frac{D}{\gcd(\d, D)} \right )\cdot  (P_\d) \in \Div(X_0(N)).
$$
The equivalence class of $C_{M, N}^D$ (modulo principal divisors) is denoted by $\cC_{M, N}^D$, which is indeed a point on $J_0(N)$ if the degree of $C_{M, N}^D$ is $0$.

\begin{rem}
If $(M, D) \in \fS(N) \setminus \fS_0(N)$, then $M=D=1$ and hence $C^{1}_{1, N}=(P_1)=\vect 1 1$. 
Therefore it cannot define a point on $J_0(N)$, so we often exclude this case. 
If $MD\neq 1$ and $M$ is a divisor of $N^\sqf$, then the degree of $C^D_{M, N}$ is $0$ because
the degree of $(P_\d)$ is $\varphi(\gcd(\d, D))$ for a divisor $\d$ of $MD$. This observation gives a hint on the fact that $\fm := (\ell, ~\cI^{D}_{M, N})$ cannot be maximal if $MD=1$ (Theorem \ref{thm : classification of Eisenstein ideals}).
\end{rem}

If $M$ is not a divisor of $N^\sqf$, i.e., $M$ and $N^{\sqq}$ are not relatively prime, then we define a cuspidal divisor $C_{M, N}^D$ using the degeneracy maps as follows. Let $\omega(M, N^{\sqq})$ be the number of prime divisors of $\gcd(M, N^{\sqq})$. We will use induction on $\omega(M, N^{\sqq})$.
\begin{itemize}
\item
If $\omega(M, N^{\sqq})=0$ then $C_{M, N}^D$ is already defined.
\item
Suppose that we have defined $C_{M, N}^D$ if $\omega(M, N^{\sqq})\leq n$.
Let $\omega(M, N^{\sqq})=n+1$, and let $p$ be a prime divisor of $\gcd(M, N^{\sqq})$. Since $p$ is a divisor of $N^{\sqq}$, we have $r=\mathrm{val}_p(N)\geq 2$. Also, since $N^{\sqq}$ is squarefree we have $\gcd(M, N^{\sqq}/p)=\gcd(M, N^{\sqq})/p$ and hence $\omega(M, N^{\sqq}/p)=n$.  
Since $(N/{p^{r-1}})^{\sqq} = N^{\sqq}/p$, the divisor $C_{M, N/{p^{r-1}}}^{D}$ is defined by assumption and hence we can define as
\begin{equation*}
C_{M, N}^{D}:=\alpha_p(N/p)^* \circ \cdots \circ \alpha_p(N/{p^{r-1}})^*(C_{M, N/{p^{r-1}}}^{D})=[N/p]_p \circ \cdots \circ [N/{p^{r-1}}]_p(C_{M, N/{p^{r-1}}}^{D}).
\end{equation*}
\end{itemize}
Therefore the definition of $C_{M, N}^D$ is completed for any $(M, D) \in \fS(N)$ by induction. The equivalence class of the divisor $C_{M, N}^D$ is denoted by $\cC_{M, N}^D$, and if $D=N^{\sqq}$, it is simply denoted by $\cC_{M, N}$.
\ms

For convenience of the reader, we give a more explicit formula for $C_{M, N}^D$. Note that if $\gcd(M, N^{\sqq})\neq 1$, then  $(M, D)$ automatically belongs to $\fS_0(N)$ and hence\footnote{By direct computation, we can prove the following. For $(M, D) \in \fS(N)$ we have $MD\neq 1$ if and only if the degree of $C_{M, N}^D$ is $0$ (without the assumption that $M$ is a divisor of $N^\sqf$).} the degree of $C_{M, N}^D$ is $0$. 
\begin{lem}\label{lem : 4.2}
Let $p$ be a prime divisor of $\gcd (M, N^{\sqq})$, $r=\mathrm{val}_p(N) \geq 2$ and $N'=N/{p^r}$.
If 
\begin{equation*}
C_{M,  N'p}^{D} = \sum_{d \mid N'} \{a(d) \cdot (P_d)+b(d) \cdot (P_{dp})\},
\end{equation*}
then we have
\begin{equation*}
C_{M, N}^D=\sum_{d \mid N'} \{ a(d)\cdot p^{r-1}\cdot (P_d)+\sum_{k=1}^r b(d)\cdot p^{\fr(k)} \cdot (P_{dp^k}) \},
\end{equation*}
where $\fr(k)=\max \{0, r-2k\}$. In particular, for $r\geq 2$ we have
\begin{equation*}
C_{1, p^r}^p = p^{r-1}\cdot (P_1)-\sum_{k=1}^r p^{\fr(k)}\cdot (P_{p^k}).
\end{equation*}
\end{lem}
\begin{proof}
Let $d$ be a divisor of $N'$. By Lemma \ref{lem : 2.1 degeneracy maps on cusps} the map $\alpha_p(N'p^t):X_0(N'p^{t+1}) \to X_0(N'p^t)$ is ramified (of ramification index $p$) at $\vect x {dp^i}$ if and only if $0 \leq i \leq t/2$. Thus, we have the formula
\begin{equation*}
 \a_p(N'p^t)^* (P_{dp^i})=\begin{cases}
 p\cdot (P_{dp^i}) & \text{ if }  i \leq  t/2, \\
(P_{dp^i}) & \text{ if }    t/2<  i < t, \\
 (P_{dp^{t+1}})+(P_{dp^t}) & \text{ if } i=t\geq 1, \\
  (P_{dp})+p\cdot (P_{d}) & \text{ if } t=i=0.
 \end{cases}
\end{equation*}
Since $C_{M, N}^D=\a_p(N'p^{r-1})^*\circ \cdots \circ \a_p(N'p)^*(C_{M, N'p}^{D})$, 
applying the maps $\a_p(N'p^t)$ inductively, we obtain the result.
\end{proof}
In all cases under consideration, we have $a(d)=-b(d)$ for any divisor $d$ of $N'$.

\ms
\subsection{The order of $\cC^D_{M, N}$}
We first compute the order of $\cC_{M, N}^D$ when $M$ is a divisor of $N^\sqf$. The remaining cases follow from this because $[N/{p^a}]_p=\alpha_p(N/{p^a})^*$ is injective for any $1\leq a \leq \mathrm{val}_p(N)$. 
\begin{thm}\label{thm : order of C_M, N}
Let $(M, D) \in \fS_0(N)$. If $M$ is a divisor of $N^\sqf$, then the order of $\cC_{M, N}^D$ is the numerator of $\frac{1}{24}\times h(M, N, D)\times \cN(M, N, D)$.
\end{thm}
If $N$ is squarefree, then the order of $\cC_{M, N}$ is already computed by the author \cite[\textsection 3]{Yoo3}. If $N$ is a prime power, it is done by Ling (cf. \cite[p. 41]{Lg97}).

\ms
Before proving this theorem, we discuss its consequence.

\begin{cor} \label{cor : the order of C_M, N}
For $(M, D) \in \fS_0(N)$, let $A=\gcd (M, N^{\sqq})$ and $B=\prod_{p \mid A} p^{\mathrm{val}_p(N)-1}$. If $A\neq 1$, then the order of $\cC_{M, N}^D$ is equal to the numerator of $\frac{1}{24}\times h(M, N/B, D)\times \cN(M, N/B, D)$.
\end{cor}
\begin{proof}
Since $A\neq 1$, we get $M>1$, so we have $(M, D) \in \fS_0(N/B)$.
Thus, by Theorem \ref{thm : order of C_M, N} the order of $\cC_{M, N/B}^{D}$ is the numerator of $\frac{h(M, N/B, D)\times \cN(M, N/B, D)}{24}$ because $\gcd(M, (N/B)^\sqq)=1$.
Since $[N/{p^a}]_p=\alpha_p(N/{p^a})^*$ is injective for any $1\leq a \leq \mathrm{val}_p(N)$ and for any prime divisor $p$ of $A$, the order of $\cC_{M, N}^D$ is equal to that of $\cC_{M, N/B}^{D}$ by definition, and hence the claim follows.  
\end{proof}

\begin{proof}[Proof of Theorem \ref{thm : order of C_M, N}]
To compute the order of a cuspidal divisor $C$ we use the same method as in \cite[\textsection 2]{Lg97}. For convenience of the reader, we review some results in \textit{loc. cit.} 

\vspace{2mm}
Let $N$ be a positive integer, and let $\cS(N)$ be the set of the divisors of $N$.
Let $$\vp=\#\cS(N)=\prod_{p \mid N} (1+\mathrm{val}_p(N))$$ 
be the number of divisors of $N$.
Let $\bd{r}=(r_\d) \in M_{\vp \times 1}(\Q)$ be a vector of size $\vp$ indexed by all the divisors of $N$. (Here, by a {\sf vector of size} $n$ we mean an $(n\times 1)$ column matrix.)
Let
\begin{equation*}
g_{\bd{r}} = \prod_{\d \in \cS(N)} \eta_{\d}^{r_\d}
\end{equation*}
be a (formal) function on the complex upper half plane, where $\eta_\d(z):=\eta(\d z)$ and $\eta$ is the Dedekind $\eta$-function. Then, $g_{\bd{r}}$ is a genuine modular function on $X_0(N)$ if and only if the following conditions are satisfied:
\begin{enumerate}
    \setcounter{enumi}{-1} 
\item
all the $r_\d$ are rational integers;
\item
$\sum_{\d \in \cS(N)} r_\d \cdot \d \equiv 0 \pmod {24}$;
\item
$\sum_{\d \in \cS(N)} r_\d  \cdot (N/\d) \equiv 0 \pmod {24}$;
\item
$\sum_{\d \in \cS(N)} r_\d  =0$;
\item
$\prod_{\d \in \cS(N)} \d^{r_\d}$ is the square of a rational number.
\end{enumerate}
(Proposition 1 of \textit{loc. cit.})

\vspace{2mm}
For divisors $d$ and $\d$ of $N$, let
\begin{equation*}
a_N(d, \d) := \frac{N}{\gcd(d, N/{d})} \cdot \frac{\gcd(d, \d)^2}{d \cdot \d} \qa b_{\bd{r}} (d):=\frac{1}{24}\sum_{\d \in \cS(N)} a_N(d, \d) \cdot r_\d.
\end{equation*}
If $\bd{r}=(r_\d) \in M_{\vp\times 1}(\Q)$ satisfies all the conditions above, then the divisor of $g_{\bd{r}}$ is supported only on the cusps, and 
\begin{equation*}
\text{div} (g_{\bd{r}})=\sum_{d \in \cS(N)} b_{\bd{r}}(d) \cdot (P_d).
\end{equation*}
Moreover, we have $\text{deg} (\text{div} (g_{\bd{r}}))=\sum_{d \in \cS(N)} b_{\bd{r}}(d) \cdot \p (\gcd(d, N/d))=0$.

Conversely, for a vector $\bd{m}=(m_d) \in M_{\vp \times 1} (\Z)$ with integral entries, if $C=\sum_{d \in \cS(N)} m_d \cdot(P_d)$ is a cuspidal divisor of degree $0$, then there is a modular function $g_{\bd{r}}$ such that $C$ is an integral multiple of the divisor of $g_{\bd{r}}$. 

\ms
Let 
\begin{equation*}
S_1(N):=\left\{ \bd{r}=(r_\d) \in M_{\vp \times 1} (\Q) : \sum_{\d \in \cS(N)} r_\d = 0 \right\}
\end{equation*}
and
\begin{equation*}
S_2(N):=\left\{ \bd{m}=(m_d) \in M_{\vp \times 1} (\Q) : \sum_{d \in \cS(N)} m_d \cdot \p(\gcd(d, N/{d}))=0 \right\}.
\end{equation*}
Let $V$ be an obvious bijection from the set of rational cuspidal divisors on $X_0(N)$
\begin{equation*}
S_c(N):=\left\{ \sum_{d \in \cS(N)} (m_d) \cdot (P_d) : (m_d) \in M_{\vp \times 1} (\Q) \right\}
\end{equation*}
to $M_{\vp \times 1}(\Q)$ defined by $V(\sum_{d \in \cS(N)} (m_d) \cdot (P_d))=(m_d)$. 
By the discussion above, we have a matrix $\Lambda(N) \in \GL_{\vp} (\Q)$, defined by  
\begin{equation*}
V^{-1}\Lambda(N)(\bd{r}) :=\sum_{d \in \cS(N)} b_{\bd{r}}(d)\cdot (P_d) \text{ for any } \bd{r}=(r_\d) \in S_1(N).
\end{equation*}
Therefore its entries are given by $\Lambda(N)_{d\d}=\frac{1}{24} a_N (d, \d)$ and 
it gives rise to a bijection from $S_1(N)$ to $S_2(N)$, and therefore its determinant is not 0 (Proposition 2 \textit{loc. cit.}). 

\vspace{2mm}
Let $C_{M, N}^D=\sum_{d \in \cS(N)} m_d \cdot (P_d) \in S_c(N)$ be a rational cuspidal divisor with $m_d \in \Z$. By the result of Ligozat \cite[\textsection 3]{Li75}, the order of $\cC_{M, N}^D$ in $J_0(N)$ is the smallest positive integer $\k$ such that the entries of $\k\Lambda(N)^{-1}V(C_{M, N}^D)$ satisfy all the conditions above\footnote{If the degree of $C_{M, N}^D$ is not $0$ then its equivalence class $\cC_{M, N}^D$ is not a point on $J_0(N)$. 
Nevertheless, we consider this case as well and compute the entries of $\Lambda(N)^{-1} V(C_{M, N}^D)$. Later, we will see that the condition (3) cannot be fulfilled for any $\k\geq 1$ in this case.}.
Below, we will explicit compute the entries of $\Lambda(N)^{-1} V(C_{M, N}^D)$ (Lemma \ref{lem : S1 entries}) and the quantities in the conditions above for $\Lambda(N)^{-1} V(C_{M, N}^D)$ (Lemma \ref{lem : 4.8}) using an inductive method. 

\vspace{2mm}
During the proof, let 
\begin{equation*}
R(M, N, D):=\Lambda(N)^{-1} V(C_{M, N}^D) \in M_{\vp \times 1}(\Q).
\end{equation*}
Note that even though the entries of $R(M, N, D)$ (in a vector form) depend on an order on the set $\cS(N)$, all the quantities in the conditions do not. So, we will often choose a special order on the set $\cS(N)$ and describe its entries (with respect to this special order) as a vector form. 

\ms
To introduce an inductive method, we fix some notation.
Let $(M, D) \in \fS_0(N)$ with $\gcd(M, N^{\sqq})=1$. Let $q$ be a prime divisor of $N$ and let $r=\mathrm{val}_q(N)\geq 1$. Let
\begin{equation*}
M':=\frac{M}{\gcd(M, q)}, D':=\frac{D}{\gcd(D, q)} \qa N':=\frac{N}{q^r}.
\end{equation*}
In other words, we have
\begin{equation*}
(M, N, D)=\begin{cases}
(M'q, N'q, D') & \text{ if } a_q(M, N, D)=1, \\
(M', N'q^r, D') & \text{ if } a_q(M, N, D)=q, \\
(M', N'q^r, D'q) & \text{ if } a_q(M, N, D)=0. \\
\end{cases}
\end{equation*}
Let $\u=\#\cS(N')$ be the number of divisors of $N'$, which is equal to $\vp/{(1+r)}$. Let $\cS(N')$, $C_{M', N'}^{D'}$, $\Lambda(N')$ and $R(M', N', D')$ be defined by replacing $M, N, D$ by $M', N', D'$. (Here, we do not require that the degree of $\cC_{M', N'}^{D'}$ is $0$. Indeed, since the determinant of $\Lambda(N')$ is not $0$ (when $N'=1$, we set $\Lambda(1)=\frac{1}{24}$), $R(M', N', D')$ is well-defined.)
Below, we will describe the entries of $\Lambda(N)^{-1}$ and
$R(M, N, D)$ using those of $\Lambda(N')^{-1}$ and $R(M', N', D')$, respectively, after choosing a suitable indexing of the set $\cS(N)$ (via the one on $\cS(N')$).

\ms
Let 
$$
b_{ij}= \frac{1}{q^r (q^2-1)} (q^{j-1}, ~q^{r+1-j}) \cdot \kappa_{ij},
$$
where
$$
\addtolength{\arraycolsep}{-2.5pt}  
\renewcommand{\arraystretch}{1.1}%
\kappa_{ij}= 
\left\{
 \begin{array}{cl}  
  q^2 & ~\text{if }  ~i=j=1~\text{ or } r+1 ; \\
  q^2+1   & ~\text{if } ~2 \leq i=j \leq r ;\\
  -q & ~\text{if } ~ | i - j | =1 ;\\
  0 & ~\text{if } ~ | i - j | \geq 2.
\end{array}
\right.
$$
Then, we have the following.
\begin{lem}\label{lem:inverseofLambda, induction}
There is a suitable indexing of the set $\cS(N)$ so that the matrix form of $\Lambda(N)^{-1}$ is
$$
\Lambda(N)^{-1} = 
\left( \begin{array}{cccc}
M_{11} & M_{12} & \cdots & M_{1(r+1)}\\
M_{21} & M_{22} & \cdots & M_{2(r+1)}\\
\vdots & \vdots & \ddots & \vdots \\
M_{(r+1)1} & M_{(r+1)2} & \cdots & M_{(r+1)(r+1)} \\
\end{array} \right)
$$
as a block matrix, where $M_{mn}$ are $\u \times \u$ matrices such that $M_{mn} = b_{mn} \times \Lambda(N')^{-1}$. 
\end{lem}
(Here, for a number $t$ and a matrix $B$, we define the matrix $tB$ as
$(tB)_{ij}:=t\cdot B_{ij}$.)

\vspace{2mm}
Before proving this lemma, we recall ``block matrix multiplication''. Let $A_{ij}$ be $\u \times \u$ matrices and let $B_k$ be $\u \times 1$ matrices for $1\leq i, j, k \leq \nu$. Also, let 
$$
X=\left( \begin{array}{cccc}
A_{11} & A_{12} & \cdots & A_{1\nu}\\
A_{21} & A_{22} & \cdots & A_{2\nu}\\
\vdots & \vdots & \ddots & \vdots \\
A_{\nu 1} & A_{\nu 2} & \cdots & A_{\nu \nu} \\
\end{array} \right)
\qa
Y=\left( \begin{array}{c}
B_{1} \\
B_{2} \\
\vdots \\
B_{\nu}\\
\end{array} \right)
$$
be an $(\u\nu)\times (\u\nu)$-matrix and a vector of size $(\u\nu)$, respectively. Then, we have
$$
X^2=\left( \begin{array}{cccc}
C_{11} & C_{12} & \cdots & C_{1\nu}\\
C_{21} & C_{22} & \cdots & C_{2\nu}\\
\vdots & \vdots & \ddots & \vdots \\
C_{\nu1} & C_{\nu2} & \cdots & C_{\nu\nu} \\
\end{array} \right)
\qa
XY=\left( \begin{array}{c}
D_{1} \\
D_{2} \\
\vdots \\
D_{\nu}\\
\end{array} \right)
$$
as block matrices, where 
\begin{equation*}
C_{ij} = \sum_{a=1}^\nu A_{ia} \times A_{aj} \qa D_k = \sum_{a=1}^\nu A_{ka} \times D_a.
\end{equation*}
The point is that if the matrices we want to multiply consist of block matrices that can be multiplied together, we can just multiply them as numbers. This simple method in linear algebra is very useful in the following computations.

\begin{proof}
By direct computation, for $d, \d \in \cS(N')$ we have
\begin{equation}
\begin{aligned}\label{eqn : a_N}
a_{N'q^r}(dq^i, \d q^j)&=\frac{N'q^r}{\gcd(dq^i, N'q^r/{dq^i})} \cdot \frac{\gcd(dq^i, ~\delta q^j)^2}{dq^i \cdot \delta q^j}\\
&=\left(\frac{N'}{\gcd(d, N'/{d})}\cdot \frac{\gcd(d, ~\delta)^2}{d \cdot \delta} \right) \times
\left( \frac{q^r}{\gcd(q^i, ~q^r/{q^i})} \cdot\frac{\gcd(q^i,~q^j)^2}{q^i \cdot q^j} \right) \\
&=a_{N'}(d, \d) \times a_{q^r}(q^i, q^j).
\end{aligned}
\end{equation}
Let $\cS(N')=\{ d_1, \dots, d_{\u} \}$. 
For an integer $1\leq a \leq \vp$, let $b$ (resp. $c$) be the quotient (resp. remainder) of $a-1$ modulo $\u$, i.e., $a-1=b\u+c$ with $0\leq c < \u$ (and $0\leq b \leq r$). We set $\d_a:=d_{c+1} \cdot q^b$. Then, we have $\cS(N)=\{ \d_1, \d_2, \dots, \d_\vp \}$ because any divisor of $N$ is of the form $d_{(c+1)} \cdot q^b$ for some $b$ and $c$ with $0\leq c < \u$ and $0\leq b \leq r$.

By Equation (\ref{eqn : a_N}), we can write $\Lambda(N)=\Lambda(N'q^r)$ as the following block matrix
$$
\Lambda(N'q^r)=\left( \begin{array}{cccc}
A_{11} & A_{12} & \cdots & A_{1(r+1)}\\
A_{21} & A_{22} & \cdots & A_{2(r+1)}\\
\vdots & \vdots & \ddots & \vdots \\
A_{(r+1)1} & A_{(r+1)2} & \cdots & A_{(r+1)(r+1)} \\
\end{array} \right),
$$
where $A_{mn} = \tau_{mn} \times \Lambda(N')$ and $\tau_{mn}=a_{q^r}(q^{m-1}, q^{n-1})=\frac{q^r}{\gcd(q^{m-1}, ~q^{r-m+1})}\cdot \frac{\gcd(q^{m-1}, ~q^{n-1})^2}{q^{m+n-2}}$. 
Since the product of two matrices
$$
\left( \tau_{mn} \right)_{\substack{1 \leq m \leq r+1 \\ 1 \leq n \leq r+1}} \quad\text{and}\quad 
\left( b_{mn} \right)_{\substack{1 \leq m \leq r+1 \\ 1 \leq n \leq r+1}}
$$
is equal to $\text{Id}_{(r+1)\times (r+1)}$ by Ling \cite[Prop. 3]{Lg97}, $(A_{mn}) \times (M_{mn}) = \text{Id}_{\u (r+1)\times \u (r+1)}=\text{Id}_{\vp \times \vp}$.
\end{proof}
Note that if $N'=1$ this is exactly the result of Ling.

\vspace{2mm}
Let $C_{M', N'}^{D'}=\sum_{i=1}^\u a(i) \cdot (P_{d_i})$ be a divisor of $X_0(N')$ (not necessarily of degree $0$), then by definition we have
\begin{equation*}
C_{M, N}^D  = \begin{cases}
\sum_{i=1}^\u a(i) \cdot (P_{d_i})-\sum_{i=1}^\u a(i) \cdot (P_{d_iq}) & \text{ if } a_q(M, N, D)=1, \\
\sum_{i=1}^\u a(i) \cdot (P_{d_i})  & \text{ if } a_q(M, N, D)=q, \\
\sum_{i=1}^\u (q-1)a(i) \cdot (P_{d_i})-\sum_{i=1}^\u a(i) \cdot (P_{d_iq})& \text{ if } a_q(M, N, D)=0.
\end{cases}
\end{equation*}
In their vector form (with respect to the indexing used in Lemma \ref{lem:inverseofLambda, induction}), as a block matrix we have 
\begin{equation*}
V(C_{M, N}^D) = \left( \begin{array}{c}
V(C_{M', N'}^{D'}) \\
-V(C_{M', N'}^{D'})
\end{array} \right)  \text{ if } a_q(M, N, D)=1.
\end{equation*}
(Note that we assume that $M$ is a divisor of $N^\sqf$.) Also, as block matrices we have
\begin{equation*}
V(C_{M, N}^{D}) = \left( \begin{array}{c}
V(C_{M', N'}^{D'}) \\
\text{\large 0} \\
\text{\large 0} \\
\vdots \\
\text{\large 0} \\
\end{array} \right) \qa
V(C_{M, N}^D) = \left( \begin{array}{c}
(q-1)V(C_{M', N'}^{D'}) \\
-V(C_{M', N'}^{D'})\\
\text{\large 0} \\
\vdots \\
\text{\large 0} \\
\end{array} \right)
\end{equation*}
if $a_q(M, N, D)=q$ and $a_q(M, N, D)=0$, respectively.

\ms
Using Lemma \ref{lem:inverseofLambda, induction}, we have the following.
\begin{lem} \label{lem : lemma 4.6}
If $a_q(M, N, D)=1$, then as a block matrix $R(M, N, D)$ is
$$
\frac{1}{q-1}\left( \begin{array}{c}
R(M', N', D') \\
-R(M', N', D')
\end{array} \right)
=\frac{1}{\cX_q(1)}\left( \begin{array}{c}
R(M', N', D') \\
-R(M', N', D')
\end{array} \right).
$$ 
Similarly, if $a_q(M, N, D)=q$ and $a_q(M, N, D)=0$, then as block matrices $R(M, N, D)$ are
\Small
$$
\frac{1}{\cY_q(r)}\left( \begin{array}{c}
q R(M', N', D') \\
- R(M', N', D') \\
\text{\small 0} \\
\text{\small 0} \\
\vdots \\
\text{\small  0} \\
\end{array} \right)
~\text{ and }~
\frac{1}{\cZ_q(r)}\left( \begin{array}{c}
q R(M', N', D') \\
-(q+1) R(M', N', D') \\
R(M', N', D') \\
\text{\small 0} \\
\vdots \\
\text{\small  0} \\
\end{array} \right),
$$
\normalsize
respectively.
\end{lem}
\begin{proof}
First, assume that $a_q(M, N, D)=1$. By definition and Lemma \ref{lem:inverseofLambda, induction}, we have
\begin{equation*}
\Lambda(N'q)=\bmat {q\Lambda(N')} {\Lambda(N')}{\Lambda(N')}
{q\Lambda(N')} \qa
\Lambda(N'q)^{-1}=\frac{1}{q^2-1}\bmat {q\Lambda(N')^{-1}} {-\Lambda(N')^{-1}}{-\Lambda(N')^{-1}}
{q\Lambda(N')^{-1}}.
\end{equation*}
Since $\Lambda(N')^{-1} V(C_{M', N'}^{D'})=R(M', N', D')$ by definition, we have
\begin{equation*}
\frac{1}{q^2-1}\bmat {q\Lambda(N')^{-1}} {-\Lambda(N')^{-1}}{-\Lambda(N')^{-1}}
{q\Lambda(N')^{-1}} \left( \begin{array}{c}
V(C_{M', N'}^{D'}) \\
-V(C_{M', N'}^{D'})
\end{array} \right) = \frac{1}{q^2-1}\left( \begin{array}{c}
(q+1)R(M', N', D') \\
-(q+1)R(M', N', D')
\end{array} \right),
\end{equation*}
which implies the first claim. The case for $a_q(M, N, D)=q$ with $r=1$ follows similarly.

Now, assume that $r\geq 2$. By Lemma \ref{lem:inverseofLambda, induction}, we have
\begin{equation*}
\Lambda(N'q^{r})^{-1} = 
\frac{1}{q^r(q^2-1)}\left(\begin{array}{ccc}
q^2\Lambda(N')^{-1} & -q^{2}\Lambda(N')^{-1}& \cdots \\
-q \Lambda(N')^{-1} & q(q^2+1)\Lambda(N')^{-1}& \cdots \\
0 & -q^{2}\Lambda(N')^{-1}& \cdots  \\  
0 & 0 & \ddots
\end{array} \right)
\end{equation*}
and hence the claim for the case with $a_q(M, N, D)=0$ holds true by the following computation:
\Small
\begin{equation*}
\left( \begin{array}{ccc}
q^{2}\Lambda(N')^{-1} & -q^{2}\Lambda(N')^{-1}& \cdots \\
-q \Lambda(N')^{-1} & q(q^2+1)\Lambda(N')^{-1}& \cdots \\
0 & -q^{2}\Lambda(N')^{-1}& \cdots  \\  
0 & 0 & \cdots \\
0 & 0 & \ddots
\end{array} \right)
\left( \begin{array}{c}
(q-1)V(C_{M', N'}^{D'}) \\
-V(C_{M', N'}^{D'}) \\
0 \\
\vdots \\
0
\end{array} \right)=\left( \begin{array}{c}
q^3 R(M', N', D') \\
-q^2(q-1)R(M', N', D') \\
q^2 R(M', N', D') \\
0 \\
\vdots
\end{array} \right).
\end{equation*}
\normalsize
Similarly, the claim for the case with $a_q(M, N, D)=q$ also follows.
\end{proof}

For a positive integer $d$ and a prime $p$, we set
$$
\addtolength{\arraycolsep}{-2.5pt}  
\renewcommand{\arraystretch}{1.1}%
x_{p}(d):= 
\left\{
 \begin{array}{cl}  
  1 & ~\text{if }  ~\mathrm{val}_p(d)=0 \\
  -1   & ~\text{if } ~\mathrm{val}_p(d) \geq 1
\end{array}
\right. 
\qa
y_{p}(d):= 
\left\{
 \begin{array}{cl}  
  p & ~\text{if }  \mathrm{val}_p(d)=0 \\
  -1   & ~\text{if } \mathrm{val}_p(d)=1\\
  0 & ~\text{if } \mathrm{val}_p(d)\geq 2. 
\end{array}
\right. 
$$
Also, we set
\begin{equation*} 
z_{p}(d):= 
\left\{
 \begin{array}{cl}  
  p & ~\text{if }  \mathrm{val}_p(d)=0 \\
  -(p+1)   & ~\text{if } \mathrm{val}_p(d)=1\\
  1   & ~\text{if } ~\mathrm{val}_p(d)=2 \\
  0  & ~\text{if } ~\mathrm{val}_p(d)\geq 3.
\end{array}
\right. 
\end{equation*}
Then, we have the following.

\begin{lem}\label{lem : S1 entries}
For a divisor $\d$ of $N$, we define
\begin{equation*}
\G(M, N, D)_\d:= \prod_{p \mid M} x_p(\d) \times \prod_{p \mid \frac{N^{\rrad}}{MD}} y_p(\d) \times \prod_{p \mid D} z_p(\d).
\end{equation*}
Then, for any $\d \in \cS(N)$ we have $R(M, N, D)_\d=\frac{24 \times \G(M, N, D)_\d}{\cN(M, N, D)}$, where $R(M, N, D)_\d$ is the entry of $R(M, N, D)$ corresponding to $\d$.
\end{lem}

\begin{proof}
Let $p(\d):=\mathrm{val}_q(\d)$ and $r(\d):=\d/{q^{p(\d)}} \in \cS(N')$.
As in Lemma \ref{lem:inverseofLambda, induction}, let $\cS(N')=\{d_1, \dots, d_\u \}$ and $\cS(N)=\cS(N'q^r)=\{ \d_1, \dots, \d_\vp \}$. (Here, for $0 \leq b \leq r$ and $1\leq c \leq \u$, we set $\d_{b\u+c}=d_c \cdot q^b$. Note that we have $p(\d_{b\u+c})=b$ and $r(\d)=d_c \in \cS(N')$.)
It is clear that
\begin{align*}
p(\d_a)=0 &\iff 1\leq a \leq \u, \\
p(\d_a)=1 &\iff \u < a \leq 2\u,\\
p(\d_a)=2 &\iff 2\u < a \leq 3\u,\\
p(\d_a)\geq 3 &\iff 3\u < a.  
\end{align*}

For a divisor $d$ of $N'$, we define the number $\G(M', N', D')_d$ as above. Then, by definition we have
\begin{equation}\label{eqn : 4.3}
\G(M, N, D)_\d=\begin{cases}
x_q(\d) \cdot \G(M', N', D')_{r(\d)} & \text{ if } a_q(M, N, D)=1,\\
y_q(\d) \cdot \G(M', N', D')_{r(\d)} & \text{ if } a_q(M, N, D)=q, \\
z_q(\d) \cdot \G(M', N', D')_{r(\d)} & \text{ if } a_q(M, N, D)=0.
\end{cases}
\end{equation}
Also, by definition we have
\begin{equation} \label{eqn : 4.4}
\frac{\cN(M, N, D)}{\cN(M', N', D')}=\begin{cases}
\cX_q(r) & \text{ if } a_q(M, N, D)=1, \\
\cY_q(r) & \text{ if }  a_q(M, N, D)=q, \\
\cZ_q(r) & \text{ if } a_q(M, N, D)=0. 
\end{cases}
\end{equation}

\ms
Suppose that $a_q(M, N, D)=0$. Then, by Lemma \ref{lem : lemma 4.6} we have
\begin{equation}\label{eqn : 4.5}
\cZ_q(r)\cdot R(M, N, D)_\d = \begin{cases}
q R(M', N', D')_{r(\d)} & \text{ if } p(\d)=0,\\
-(q+1) R(M', N', D')_{r(\d)} & \text{ if } p(\d)=1,\\
 R(M', N', D')_{r(\d)} & \text{ if } p(\d)=2,\\
0 & \text{ if } p(\d) \geq 3.
\end{cases}
\end{equation}
Suppose further that the above formula holds for the level $N'$. In other words, we assume that
\begin{equation*}
R(M', N', D')_d=\frac{24\times \G(M', N', D')_{d}}{\cN(M', N', D')} \text{ for all }~d \in \cS(N').
\end{equation*}
Then by Equations (\ref{eqn : 4.3}), (\ref{eqn : 4.4}) and (\ref{eqn : 4.5}), we have
\begin{equation*}
R(M, N, D)_\d = \frac{24\times \G(M, N, D)_\d}{\cN(M, N, D)} \text{ for all }~\d \in \cS(N),
\end{equation*}
as desired. Similar computations imply the remaining cases as well. Hence we obtain the result for the level $N$.

Note that when $M'=N'=D'=1$ then $C_{1, 1}=(P_1)$ and $\Lambda(1)=\frac{1}{24}$. Thus, $R(1,1,1)=24$ and the claim is vacuously true.\footnote{Indeed, if $N=p^k$ for some $k\geq 1$ then we can directly compute $R(M, p^k, D)$ via the same method as above, which is already obtained by Ling \cite{Lg97}.}
Thus, applying the same method inductively, we obtain the result.
\end{proof}

\vspace{3mm}
Now, we can compute all the quantities in the conditions listed above. 

\begin{lem}\label{lem : 4.8}
Let $\bd{r}=(\G(M, N, D)_\d) \in S_1(N)$. Then, we have the following.
\begin{enumerate}
\item
We have
\begin{equation*}
\sum_{\d \in \cS(N)} \G(M, N, D)_\d \cdot \d = \begin{cases}
(-1)^{\o(M)} \times \cN(M, N, D) & \text{ if } M=N,\\
0 & \text{ otherwise}.
\end{cases}
\end{equation*}
\item
We have
\begin{equation*}
\sum_{\delta \in \cS(N)} \G(M, N, D)_{\delta} \cdot (N/\delta)=\cN(M, N, D) \times \prod_{p \mid D} (p-1).
\end{equation*}

\item
We have
$$
\sum_{\d \in \cS(N)} \G(M, N, D)_\d  = 0.
$$
\item
We have
\begin{equation*}
\mathrm{val}_q \left(\prod_{\d \in \cS(N)} \d^{\G(M, N, D)_\d} \right)=\begin{cases}
-\prod_{p \mid \frac{N^{\rrad}}{q}} (p-1) & \text{ if } M=q \text{ and } D=1,  \\
-\prod_{p \mid N^{\rrad}} (p-1) & \text{ if } M=1 \text{ and } D=q, \\
0 &\text{ otherwise}.
\end{cases}
\end{equation*}
\end{enumerate}
\end{lem}

\ms
Note that since we assume that $\bd{r} \in S_1(N)$, we have $(M, D) \in \fS_0(N)$ and $N>1$.
\begin{proof}
We use the same notation as in the proof of Lemma \ref{lem : S1 entries}.
Also, for ease of notation we set $\G_\d=\G(M, N, D)_\d$ and $\D_d=\G(M', N', D')_d$.
Below, we will use Equation (\ref{eqn : 4.3}) without referring to it. Also, we use the fact that 
\begin{equation*}
\cS(N) = \coprod_{a=0}^r \cS(N')_a,
\end{equation*}
where $\cS(N')_a:=\{ dq^a : d \in \cS(N') \}$. 

\begin{enumerate}
\item
If $a_q(M, N, D)=q$, then we have
$$
\sum_{\delta \in \cS(N)} \G_{\delta} \cdot \delta=
\sum_{d \in \cS(N')}  (\G_{d} \cdot d+\G_{d q} \cdot d q)=\sum_{d \in \cS(N')}((q\D_d)\cdot d+(-\D_d)\cdot dq)=0.
$$
If $a_q(M, N, D)=0$, then we have
\begin{equation*}
\begin{split}
\sum_{\delta \in \cS(N)} \G_{\delta} \cdot \delta&=
\sum_{d \in \cS(N')}  (\G_{d} \cdot d+\G_{d q} \cdot d q+ \G_{dq^2} \cdot dq^2+\sum_{a=3}^r \G_{dq^a} \cdot dq^a)\\
&=\sum_{d \in \cS(N')}  ((q\D_d) \cdot d+(-(q+1)\D_d) \cdot d q+ (\D_d) \cdot dq^2+0)=0.
\end{split}
\end{equation*}
Thus, if there is a prime divisor $q$ of $N$ with $a_q(M, N, D) \neq 1$, then the sum is $0$. 
If $a_q(M, N, D)=1$ for all prime divisors $q$ of $N$, then we have $M=N$ and $D=1$
because we assume that $M$ is a divisor of $N^\sqf$.
Moreover, we have
$$
\sum_{\delta \in \cS(N)} \G_{\delta} \cdot \delta=
\sum_{d \in \cS(N')} (\G_d \cdot d + \G_{dq} \cdot dq)=\sum_{d \in \cS(N')} (\D_d \cdot d + (-\D_d)\cdot dq)
=(1-q)\cdot \sum_{d \in \cS(N')} \D_d \cdot d. 
$$
Applying the same method inductively, we obtain
$$
\sum_{\delta \in \cS(N)} \G_{\delta} \cdot \delta=\prod_{p \mid M} (1-p)=(-1)^{\o(M)} \times \cN(M, N, D).
$$

\item
Similarly as above, if $a_q(M, N, D)=1$ then we have
\begin{equation*}
\sum_{\delta \in \cS(N)} \G_{\delta} \cdot (N/\delta)
=\sum_{d \in \cS(N')}  (\G_{d} \cdot (N'/d)q+\G_{d q} \cdot (N'/d))
=(q-1) \cdot \sum_{d \in \cS(N')}  \D_{d} \cdot (N'/d).
\end{equation*}
If $a_q(M, N, D)=q$, then we have
\begin{equation*}
\begin{split}
\sum_{\delta \in \cS(N)} \G_{\delta} \cdot (N/\delta)
&=\sum_{d \in \cS(N')}  (\G_{d} \cdot (N'/d)q^r+\G_{d q} \cdot (N'/d)q^{r-1}+\cdots)\\
&=q^{r-1}(q^2-1) \cdot \sum_{d \in \cS(N')}  \D_{d} \cdot (N'/d).
\end{split}
\end{equation*}
If $a_q(M, N, D)=0$, then we have\footnote{The term $q^{r-2}(q^2-1)(q-1)$ also appears in the computation of the residues in Lemma \ref{lem : 5.8 new}.}
\begin{equation*}
\begin{split}
\sum_{\delta \in \cS(N)} \G_{\delta} \cdot (N/\delta)
&=\sum_{d \in \cS(N')}  (\G_{d} \cdot (N'/d)q^r+\G_{d q} \cdot (N'/d)q^{r-1}+\G_{dq^2} \cdot (N'/d)q^{r-2}+\cdots )\\
&=q^{r-2}(q-1)(q^2-1)\cdot \sum_{d \in \cS(N')}  \D_{d} \cdot (N'/d).
\end{split}
\end{equation*}
Thus, applying the same method inductively we get 
\begin{equation*}
\sum_{\delta \in \cS(N)} \G_{\delta} \cdot (N/\delta)=\cN(M, N, D) \times \prod_{p \mid D} (p-1).
\end{equation*}
\item
If $a_q(M, N, D)=1$, then we have
\begin{equation*}
\sum_{\delta \in \cS(N)} \G_{\delta}
=\sum_{d \in \cS(N')} (\G_\d+\G_{dq})=0.
\end{equation*}
Also, if $a_q(M, N, D)=0$, then we have
\begin{equation*}
\sum_{\delta \in \cS(N)} \G_{\delta}
=\sum_{d \in \cS(N')} (\G_\d+\G_{dq}+\G_{dq^2}+\cdots)=0.
\end{equation*}
On the other hand, if $a_q(M, N, D)=q$, then we have
\begin{equation*}
\sum_{\delta \in \cS(N)} \G_{\delta}
=\sum_{d \in \cS(N')} (\G_\d+\G_{dq}+\cdots)
=\sum_{d \in \cS(N')} (q\D_\d+(-\D_d)+0)=(q-1) \cdot \sum_{d \in \cS(N')} \D_d.
\end{equation*}
Thus, if $M=D=1$, then applying the same method inductively, we get
$$\sum_{\delta \in \cS(N)} \G_{\delta}=\prod_{p \mid N^{\rrad}} (p-1).$$
Note that if $M=D=1$ then the divisor $C_{M, N}^D$ is of degree $1$, and hence $\bd{r}=(\G_\d) \not\in S_1(N)$. Therefore we have $\sum_{\delta \in \cS(N)} \G_{\delta}=0$ if $\bd{r} \in S_1(N)$.
\item

Note that we have
\begin{equation*}
\prod_{\d \in \cS(N)} \d^{\G_\d}=\prod_{d  \in \cS(N')} ( d^{\G_d} \times (dq)^{\G_{d q}} \times (dq^2)^{\G_{d q^2}} \times \cdots )= A \times \prod_{d \in \cS(N')} q^{(\G_{dq}+2 \G_{dq^2}+\cdots)},
\end{equation*}
where $A$ is a positive integer prime to $q$. Note that
\begin{equation*}
\G_{dq}+2\G_{dq^2}+ \cdots = \begin{cases}
-\D_d & \text{ if } a_q(M, N, D)=1 \text{ or } q, \\
(1-q)\D_d & \text{ if } a_q(M, N, D)=0.
\end{cases}
\end{equation*}
By (3), $\sum_{d \in \cS(N')} \D_{d}=0$ if the degree of $C_{M', N'}^{D'}$ is $0$, and   
$\sum_{d \in \cS(N')} \D_d=\prod_{p \mid N'} (p-1)$ otherwise. Moreover, the degree of $C_{M', N'}^{D'}$ is not $0$ if and only if $MD=q$. Therefore, we have
\begin{equation*}
\mathrm{val}_q \left(\prod_{\d \in \cS(N)} \d^{\G_\d} \right)=\begin{cases}
-\prod_{p \mid \frac{N^{\rrad}}{q}} (p-1) & \text{ if } M=q \text{ and } D=1,  \\
-\prod_{p \mid N^{\rrad}} (p-1) & \text{ if } M=1 \text{ and } D=q, \\
0 &\text{ otherwise}.
\end{cases}
\end{equation*}
\end{enumerate}
\end{proof}

\begin{rem}\label{rem : 4.9}
Recall the definition of $h=h(M, N, D)$. If one of the following holds then we set $h=2$, and $h=1$ otherwise:
\begin{enumerate}[(i)]
\item
$N=M$ is a prime.
\item
$N=2^k M$ with $M$ an odd prime for $k\geq 1$, and $D=1$.
\item
$M=1$, $N=2^k$ for $k\geq 2$ and $D=2$.
\end{enumerate}
By (4) in the lemma above, $h=2$ if and only if $\prod_{\d \in \cS(N)} \d^{\G_\d}$ is not a square of a rational number. Thus, if $h=1$ then the condition (4) is automatically satisfied.
\end{rem}

Finally, we compute the order of $\cC_{M, N}^D$. Assume that $(\k R(M, N, D)_\d) \in S_1(N)$ satisfies all the conditions above. 
We want to find the smallest positive integer $\k$, which is the order of $\cC_{M, N}^D$.
By the condition (0) we have 
$$\k=\frac{\cN(M, N, D)}{24} \times \g$$ 
for some integer $\g\geq 1$ because $\G(M, N, D)_\d=(-1)^{\o(N^{\rrad}/D)}$ where $\d=N^{\rrad} \cdot D$. 
Since $\k$ is also an integer, $\g \cdot \cN(M, N, D)$ is a multiple of $24$, i.e., 
\begin{equation*}
\g=\frac{24}{\gcd(24, ~\cN(M, N, D))} \times g 
\end{equation*}
for some integer $g\geq 1$. Therefore to find the smallest $\k$, it suffices to find the minimum of $g$.
By Lemma \ref{lem : S1 entries}, we have $\k R(M, N, D)_\d= \g \G(M, N, D)_\d$ and hence $(\g\G(M, N, D)_\d) \in S_1(N)$ satisfies all the conditions above. Now, we use Lemma \ref{lem : 4.8}. Since $\g \cdot \cN(M, N, D)$ is a multiple of $24$, the conditions (1) and (2) are automatically fulfilled. Also, the condition (3) always hold. 
By Remark \ref{rem : 4.9}, if $h=1$ then the condition (4) is also automatically true and hence the minimum of $g$ is 1. So, the order of $\cC_{M, N}^D$ is
\begin{equation*}
 \frac{\cN(M, N, D)}{24} \times \frac{24}{\gcd(24, ~\cN(M, N, D))} \times 1 =\frac{\cN(M, N, D)}{\gcd(24, ~\cN(M, N, D))},
 \end{equation*}
which is the numerator of $\frac{\cN(M, N, D)}{24}$.

\ms
Now, let $h=2$. Then, $\prod_{\d \in \cS(N)} \d^{\G_\d}$ is not a square, so $\g$ must be even by the condition (4). If $G(M, N, D):=\frac{24}{\gcd(24,~\cN(M, N, D))}$ is already even, then the minimum of $g$ is $1$ as above, and otherwise it is $2$.

{\bf Case (i) } $M=N=q$ is a prime and $D=1$.  

By direct computation, we have
\begin{equation*}
\frac{\cN(q, q, 1)}{24}=\frac{q-1}{24} \qa G(q, q, 1)=\begin{cases}
\text{odd} & \text{ if } q \equiv 1 \pmod 8 \\
\text{even} & \text{ if } q \not\equiv 1 \pmod 8.
\end{cases}
\end{equation*}
Thus, the order of $\cC_{q, q}$ is the numerator of $\frac{q-1}{12}=\frac{\cN(q, q, 1)}{12}$.

{\bf Case (ii) } $M=q$ is an odd prime, $N=2^k q$ for some $k\geq 1$ and $D=1$.
 
By direct computation, we have $\cN(q, 2^k q, 1)=2^{k-1}\cdot 3(q-1)$ and $G(q, 2^k q, 1)=2^a$ for some $0\leq a \leq 3$.
Moreover, we have
\begin{align*}
G(q, 2q, 1)=1 &\iff q\equiv 1 \pmod 8\\
G(q, 4q, 1)=1 &\iff q\equiv 1 \pmod 4\\
G(q, 2^k q, 1)=1 & \quad \text{ if } k \geq 3.
\end{align*}
Therefore the order of $\cC_{q, 2^k q}^1$ is the numerator of $2^{k-3}(q-1)=\frac{\cN(q, 2^k q, 1)}{12}$.

{\bf Case (iii) } $M=1$, $D=2$ and $N=2^k$ for some $k\geq 2$. 

If $k=2$ or $3$ then the genus of $X_0(2^k)$ is 0 and hence the order of $\cC_{1, 2^k}^2$ is automatically $1$. By direct computation, we have 
$\cN(1, 2^k, 2)=3\cdot 2^{k-2}$. Also we have 
$G(1, 2^k, 2)=1$ if $k\geq 5$, and $G(1, 16, 2)=2$.
Therefore the minimum of $g$ is $2$ for all $k\geq 4$ and hence the order of $\cC_{1, 2^k}^2$ for $k\geq 2$ is the numerator of $2^{k-4}=\frac{\cN(1, 2^k, 2)}{12}$.

In summary, if $h=2$ then the order of $\cC_{M, N}^D$ is the numerator of 
$$
\frac{\cN(M, N, D)}{12}=\frac{h\times \cN(M, N, D)}{24},
$$
as claimed.
\end{proof}

\ms
\subsection{The Hecke action on $\br{\cC^D_{M, N}}$}
As above, let $(M, D) \in \fS_0(N)$, i.e., $M$ is a divisor of $N^{\rrad}$ and $D$ is a divisor of $N^{\sqq}$ prime to $D$ satisfying $MD \neq 1$. 
\begin{thm}\label{thm : hecke action on C_M, N}
The ideal $\cI_{M, N}^D$ annihilates $\br {\cC_{M, N}^D}$. In other words, for any prime $p$, we have
\begin{equation*}
T_p(\cC_{M, N}^D)=a_p(M, N, D)\cdot \cC_{M, N}^D.
\end{equation*}

\end{thm}
\begin{proof}
We first assume that $M$ divides $N^\sqf$, and set notation as follows.
Let $p$ be a prime number and let $r:=\mathrm{val}_p(N)$. 
Let $d$ be a divisor of $N/p^r$ and let $t:=\gcd(d, N/d)$. 
Let 
\begin{equation*}
\cA=\{ n \in \Z : 1\leq n \leq t, ~\gcd(n, t)=1 \} \qa \cB = \{ n \in \Z : 1 \leq n \leq pt, ~\gcd(n, pt)=1 \}.
\end{equation*}
For each $n\in \cA \cup \cB$, let
\begin{equation*}
x(n) =\begin{cases}
x & \text{ if } \gcd(n, p)=1, \\
x+d & \text{ otherwise}. \\
\end{cases}
\end{equation*}
(So, we have $x(n) \in \nabla(dp)$ and $x(n) \equiv n \pmod {t}$.) For each $n \in \cA$, let $\cB_n$ be the subset of $\cB$ consisting of elements congruent to $n$ modulo $t$. (Note that $\# \cB_n=p-1$ for any $n \in \cB_n$.)
Finally, we choose a prime $p^*$ so that $p^*p \equiv 1 \pmod d$, which is possible because $\gcd(p, d)=1$.

By the discussion in Section \ref{sec : the cusp}, we have $(P_d)=\sum_{n\in \cA} \vect {x(n)} d$ and $(P_{dp^r})=\sum_{n \in \cA} \vect {x(n)} {dp^r}$. Also, since $\gcd(p p^*, t)=1$, we get
 \begin{equation} \label{eqn : sum P_d}
(P_d)=\sum_{n \in \cA} \vect {x(n)} d = \sum_{n \in \cA} \vect {p x(n)}d=\sum_{n \in \cA} \vect {p^* x(n)}d
\end{equation}
and
\begin{equation} \label{eqn : sum P_pd}
(P_{dp})=\begin{cases}
\sum_{n \in \cA} \vect {x(n)}{dp}=\sum_{n \in \cA} \vect {p^* x(n)}{dp} &\text{ if } r=1,\\
\sum_{n \in \cB} \vect {x(n)}{dp}=\sum_{n \in \cA} \sum_{y \in \cB_n} \vect {y}{dp} &\text{ if } r\geq 2.
\end{cases}
\end{equation}

Let $T_p=\beta_p(N)_* \circ \alpha_p(N)^* : \Div (X_0(N)) \rightarrow \Div (X_0(N))$, which induces the map $T_p$ on $J_0(N)$.
Below, we use Equations (\ref{equation : alpha cusp}) and (\ref{equation : beta cusp}) without referring to them.
\begin{itemize}
\item
Case 1 : $r=0$. 
For $n \in \cA$, we have 
$$
\beta_p(N)_*\circ \alpha_p(N)^*\vect {x(n)}{d} = \beta_p(N)_* \left( p \vect{x(n)}{d}+\vect{p^* x(n)}{dp} \right) = p \vect {px(n)}{d}+\vect {p^* x(n)}{d}.
$$
Therefore $\Delta_p ((P_d))=(p+1) \cdot (P_d)$ by Equation (\ref{eqn : sum P_d}).
\item
Case 2 : $r=1$. Let $n \in \cA$. Then, we have
$$
\beta_p(N)_*\circ \alpha_p(N)^*\vect {x(n)}{d}=\beta_p(N)_* \left( p \vect{x(n)}{d} \right) = p \vect {px(n)}{d}.
$$
Let $y \in \cB_n$. Since $y \equiv x(n) \pmod t$, $\vect y d = \vect {x(n)}{d}$ and $\vect y {dp}=\vect {x(n)}{dp}$
as cusps of $X_0(N)$. Therefore, we get
\begin{equation*}
\beta_p(N)_*\circ \alpha_p(N)^*\vect {x(n)}{dp}=\beta_p(N)_* \left( \sum_{y \in \cB_n} \vect{y}{dp} +\vect {p^* x(n)}{dp^2} \right) 
= (p-1) \vect {x(n)}{d}+ \vect{p^* x(n)}{dp}.
\end{equation*}
Thus, $\Delta_p((P_d))=p \cdot (P_d)$ and $\Delta_p((P_{dp}))=(p-1)\cdot (P_d)+(P_{dp})$ by Equations (\ref{eqn : sum P_d}) and (\ref{eqn : sum P_pd}).

\item
Case 3 : $r\geq 2$. Let $n \in \cA$. Then, we have
$$
\beta_p(N)_*\circ \alpha_p(N)^*\vect {x(n)}{d}=\beta_p(N)_* \left( p \vect{x(n)}{d} \right) = p \vect {px(n)}{d}.
$$
Also, for $y \in \cB_n$, we have
$$
\beta_p(N)_*\circ \alpha_p(N)^*\vect {y}{dp}=\beta_p(N)_* \left( p \vect{y}{dp} \right) = p\vect {y}{d}=p\vect {x(n)} d
$$
since $y \equiv x(n) \pmod t$.
Therefore $\Delta_p((P_d))=p\cdot (P_d)$ and $\Delta_p((P_{pd}))=p(p-1)\cdot (P_d)$ by Equations (\ref{eqn : sum P_d}) and (\ref{eqn : sum P_pd}).
\end{itemize}
 
Thus, we have the following.
\begin{itemize}
\item
Case 1 : $r=0$. Since $\Delta_p((P_d))=(p+1)\cdot (P_d)$ for any $d \mid N$, we have 
$$
T_p(\cC_{M, N}^D)=(p+1) \cdot \cC_{M, N}^D.
$$

\item
Case 2 : $r=1$ and $p \mid M$. Since $\Delta_p( (P_d)-(P_{dp})) = (P_d)-(P_{dp})$ for any $d \mid \frac{N}{p}$, we have 
$$
T_p(\cC_{M, N}^D)=\cC_{M, N}^D.
$$
\item
Case 3 : $r \geq 2$ and $p \mid D$. 
Since $\Delta_p((p-1)\cdot (P_d)-(P_{dp}))=0$ for any $d \mid \frac{N}{p^r}$, we have 
$$
T_p(\cC_{M, N}^D)=0.
$$

\item
Case 4 : $r\geq 1$ and $p \nmid MD$. Since $\Delta_p((P_d))=p \cdot (P_d)$ for any $d \mid \frac{N}{p^r}$, we have 
$$
T_p(\cC_{M, N}^D)=p \cdot \cC_{M, N}^D.
$$
\end{itemize}
Therefore the claim follows if $M$ divides $N^\sqf$.

If $\gcd(M, N^{\sqq}) \neq 1$, then by Proposition \ref{prop : the maps [p]} and Remark \ref{rem : eigenvalue}, 
we can reduce to the cases above, and hence we obtain the result.
\end{proof}

\subsection{Proof of Theorem \ref{thm:corollary}}\label{sec : the map [N] on the cuspidal divisors}
In this subsection, we prove Theorem \ref{thm:corollary}.
\begin{lem}\label{lem : 4.11}
We have $[N]_p^\bullet(\cC_{M, N})=\cN \cdot \cC_{M', Np}$, i.e., 
\begin{equation}
\begin{cases}
[N]^-_p(\cC_{M, N})=(p+1) \cdot \cC_{M/p, Np} & \text{if } p \mid M,\\
[N]^+_p(\cC_{M, N})=\cC_{M, Np} & \text{if } p \mid \frac{N^{\sqf}}{M}, \\
[N]_p(\cC_{M, N})=p \cdot \cC_{M, Np} & \text{if } p \mid N^{\sqq}.
\end{cases}
\end{equation}
\end{lem}
\begin{proof}
This follows from a similar computation as in the proof of Theorem \ref{thm : hecke action on C_M, N}. 

Let $r=\mathrm{val}_p(N)$ and let $d$ be a divisor of $N'=N/{p^r}$. From the discussion in Section \ref{sec : the cusp}, since we assume that $r\geq 1$, we have
\begin{equation*}
\a_p(N)^*((P_d))=p\cdot (P_d) \qa \b_p(N)^*((P_d))=(P_d)+(P_{dp}).
\end{equation*}
Also, we have
\begin{equation*}
\a_p(N)^*((P_{dp}))=
\begin{cases}
(P_{dp})+(P_{dp^2}) & \text{ if } r=1,\\
p\cdot (P_{dp}) & \text{ if } r\geq 2
\end{cases}
 \qa \b_p(N)^*((P_{dp}))=\begin{cases}
p\cdot (P_{dp^2}) & \text{ if } r \leq 2,\\
(P_{dp^2})& \text{ if } r\geq 3.
\end{cases}
\end{equation*}
Let $C_{M', N'}=\sum_{d \mid N'} a(d)\cdot (P_d)$, where $M'=\frac{M}{(M, ~p)}$. (This divisor may not be of degree $0$ though.) Then by definition we have
\begin{equation*}
C_{M, N}=\begin{cases}
\sum_{d \mid N'} a(d)\cdot [(P_d) - (P_{dp})] & \text{ if } p \mid M,\\
\sum_{d \mid N'} a(d)\cdot (P_d) & \text{ if } p \mid \frac{N^{\sqf}}{M}, \\
\sum_{d \mid N'} a(d) \cdot [(p-1)\cdot (P_d)-(P_{dp})]  & \text{ if } p \mid N^{\sqq}
\end{cases}
\end{equation*}
and
\begin{equation*}
C_{M', Np}=\sum_{d \mid N'}  a(d)\cdot [(p-1)\cdot (P_d)-(P_{dp})] .
\end{equation*}
Since $r=1$ for a prime divisor $p$ of $N^\sqf$ and $r\geq 2$ otherwise, we have
\begin{equation*}
\a_p(N)^*(C_{M, N})=\begin{cases}
\sum_{d \mid N'} a(d)\cdot [p\cdot (P_d) - (P_{dp})-(P_{dp^2})] & \text{ if } p \mid M,\\
\sum_{d \mid N'} p\cdot a(d)\cdot (P_d) & \text{ if } p \mid \frac{N^{\sqf}}{M}, \\
\sum_{d \mid N'} p\cdot a(d)\cdot [(p-1)\cdot (P_d)-(P_{dp})]  & \text{ if } p \mid N^{\sqq},
\end{cases}
\end{equation*}
and
\begin{equation*}
\b_p(N)^*(C_{M, N})=\begin{cases}
\sum_{d \mid N'} a(d)\cdot [(P_d) + (P_{dp})-p\cdot (P_{dp^2})] & \text{ if } p \mid M,\\
\sum_{d \mid N'} a(d)\cdot [(P_d) + (P_{dp})] & \text{ if } p \mid \frac{N^{\sqf}}{M}, \\
\sum_{d \mid N'} a(d)\cdot [(p-1)\cdot (P_d)+(p-1)\cdot (P_{dp})-p\cdot (P_{dp^2})]  & \text{ if } p \mid N^{\sqq} \text{ and } r=2\\
\sum_{d \mid N'} a(d)\cdot [(p-1)\cdot (P_d)+(p-1)\cdot (P_{dp})-(P_{dp^2})]  & \text{ if } p \mid N^{\sqq} \text{ and } r\geq 3.
\end{cases}
\end{equation*}
By the definition of $[N]_p^\bigstar$, the result follows.
\end{proof}

Thus, we have proved the first claim. 
Now, let $m$ and $n$ be the orders of $[N]_p^\bullet(\cC_{M, N})$ and $\cC_{M, N}$, respectively. Then, we have a canonical isomorphism
\begin{equation*}
K \simeq \< \cC_{M, N} \> [n/m].
\end{equation*}
Since $\<\cC_{M, N}\>[k]=\<\cC_{M, N}\>[\gcd(n, k)]$, the proof of Theorem \ref{thm:corollary} is completed by the following lemma.
(Note that the last claim immediately follows from the previous claim.)
\begin{lem}
Let $h=\frac{h(M, N, N^{\sqq})}{h(M', Np, (Np)^{\sqq})}$. Then $h$ is either $1$ or $2$. Also, $h=2$ if and only if either $(M, N)=(q, q)$ or $(M, N)=(q, 2q)$ for an odd prime $q$.
In any case, the order of $\cN\cdot \cC_{M', Np}$ is $\frac{n}{\gcd(n, h)}$.
\end{lem}
\begin{proof}
By Lemma \ref{lem : 4.11}, we have $[N]_p^\bullet(\cC_{M, N})=\cN \cdot \cC_{M', Np}$. Thus, we are going to prove that the order of $\cN \cdot \cC_{M', Np}$ is $\frac{n}{\gcd(n, h)}$. 
In the course of the proof, the claims for $h$ will follow.

First, assume that $h(M, N, N^{\sqq})=1$. Since $p^2$ divides $Np$ (and $D=N^{\sqq}$), we always have $h(M', Np, (Np)^{\sqq})=1$ and hence $h=1$.  By direct computation, we have $\cN(M', Np, (Np)^{\sqq})=\cN \cdot \cN(M, N, N^{\sqq})$, and  the order of $\cC_{M', Np}$ is
\begin{equation*}
\frac{\cN \cdot \cN(M, N, N^{\sqq})}{\gcd(24, \cN \cdot \cN(M, N, N^{\sqq}))} (=:\cM)
\end{equation*}
by Theorem \ref{thm : order of C_M, N}. Since the order of $\cC_{M', Np}$ is $\cM$, the order of $\cN \cdot \cC_{M', Np}$ is 
$\frac{\cM}{\gcd(\cN, \cM)}$. Therefore by Lemma \ref{lem : ABC} (take $A=\cN(M, N, N^{\sqq})$, $B=\cN$ and $C=24$), the order of $\cN \cdot \cC_{M', Np}$ is 
\begin{equation*}
\frac{\cN(M, N, N^{\sqq})}{\gcd(24, \cN(M, N, N^{\sqq}))},
\end{equation*}
which is equal to $n$ by Theorem \ref{thm : order of C_M, N}. 

Next, we assume that $h(M, N, N^{\sqq})=2$ and $N=2^k$ for some $k\geq 1$. Then, we have $(M, N)=(2, 2)$ or $(1, 2^k)$ with $k\geq 2$. 
Since $p$ is a divisor of $N$, we have $p=2$ and hence $h(M', Np, 2)=h(1, 2^{k+1}, 2)=2$ and $h=1$. 
By direct computation, the order of $\cN \cdot \cC_{M', Np}$ is equal to the numerator of $2^{k-4}$, which is equal to $n$.

Finally, we assume that $h(M, N, N^{\sqq})=2$ and $N$ is not a power of $2$. Let $q$ be an odd prime divisor of $N$. Then, $(M, N)=(q, q)$ or $(q, 2q)$. By direct computation, we have
\begin{equation*}
h(1, q^2, q)=h(q, 4q, 2)=h(1, 2q^2, q)=1,
\end{equation*}
and hence $h=2$. If $(M, N)=(q, q)$ then the order of $(q+1)\cdot \cC_{1, q^2}$ is the numerator of $\frac{q-1}{24}$, which is $\frac{n}{\gcd(n, 2)}$ because $n$ is the numerator of $\frac{q-1}{12}$. 
If $(M, N)=(q, 2q)$ then either $p=q$ or $p=2$. By direct computation, the orders of $(q+1) \cdot \cC_{1, 2q^2}$ and $\cC_{q, 4q}$ are equal to the numerator of $\frac{q-1}{8}$, which is $\frac{n}{\gcd(n, 2)}$ because $n$ is the numerator of $\frac{q-1}{4}$. 
\end{proof}

\begin{lem}\label{lem : ABC}
Let $A, B$ and $C$ be positive integers. Then, we have
\begin{equation*}
\frac{AB}{\gcd(C, AB)}\times \frac{1}{\gcd(B, \frac{AB}{\gcd(C, AB)})}=\frac{A}{\gcd(C, A)}.
\end{equation*}
\end{lem}
\begin{proof}
Let $X=\frac{AB}{\gcd(C, AB)}$, $Y=\gcd(B, X)$ and $Z=\frac{A}{\gcd(C, A)}$.
To prove the equality, it suffices to show that $\mathrm{val}_p(X/Y)=\mathrm{val}_p(Z)$ for each prime divisor $p$ of $AB$.
Therefore, we can assume that $A=p^a, B=p^b$ and $C=p^c$ for some integers $a, b, c \geq 0$. Note that
\begin{equation*}
\mathrm{val}_p (X) = a+b-\min \{a+b, c\}, \mathrm{val}_p (Y)=\min \{b,  \mathrm{val}_p(X)\} \text{ and } \mathrm{val}_p(Z)=a-\min \{a, c\}.
\end{equation*}
Suppose that $a+b\geq c$. Then $\mathrm{val}_p(X)=a+b-c$ and $b\leq b+(a-c)$. Thus, we have 
 $\mathrm{val}_p(XY)=\max \{0, a-c \}=\mathrm{val}_p(Z)$. Suppose that $a+b < c$. Then, $a < c$ and hence we have
 $\mathrm{val}_p(X)=\mathrm{val}_p(Y)=\mathrm{val}_p(Z)=0$ as desired.
\end{proof}

\ms
\section{The Eisenstein series $\cE^D_{M, N}$} \label{sec : Eisenstein series}
In this section, for $(M, D) \in \fS_0(N)$ we define an Eisenstein series $\cE^D_{M, N}$ and compute its residues at various cusps.

\subsection{Definition}
Let $\cE_{p, p}$ be the unique Eisenstein series of weight $2$ for $\Gamma_0(p)$ whose $q$-expansion\footnote{This $q$-expansion is the power series $e'$ on \cite[p. 78]{M77}.} is
\begin{equation*}
1-p-24\sum_{m =1}^\infty \s'(m) \cdot q^m,
\end{equation*}
where $\s'(m)$ is the sum of the divisors of $m$ which are prime to $p$.
For $k\geq 2$, let $\cE_{1, p^k}$ be the Eisenstein series of weight $2$ for $\Gamma_0(p^k)$ defined by
$$
\cE_{1, p^k}(z):=[p^{k-1}]_p \circ \cdots \circ [p^2]_p \circ [p]^-_p (\cE_{p, p})(z)=p\cdot (\cE_{p, p}(z)-\cE_{p, p}(pz)).
$$ 

Before defining an Eisenstein series $\cE_{M, N}^D$ for any $(M, D) \in \fS_0(N)$, we prove the following lemma.
\begin{lem}\label{lem : 5.1}
Let $(M, N) \in \fS_0(N)$. Then the following holds.
\begin{enumerate}
\item
If $N=p$ is prime, then $(M, D)= (p, 1)$.
\item
If $N=p^k$ is a prime power for $k\geq 2$, then $(M, D)$ is either $(p, 1)$ or $(1, p)$.
\item
Assume that $N$ is not a prime power, i.e., $N$ is divisible by at least two primes. Then, there is a prime divisor $q$ of $N$ such that $(M', D') \in \fS_0(N')$, where 
\begin{equation*}
M':=\frac{M}{\gcd(M, q)}, D':=\frac{D}{\gcd(D, q)} \qa N':=\frac{N}{q^{\mathrm{val}_q(N)}}.
\end{equation*}
\end{enumerate}
\end{lem}
\begin{proof}
The first two claims easily follow by definition. 

Assume that $N$ is divisible by at least two primes, i.e., $N^{\rrad}$ is divisible by at least two primes. Since $(M', D') \in \fS(N')$ by definition for any prime divisor $q$ of $N$, 
it suffices to show that $M'D' \neq 1$ for some divisor $q$ of $N$. 
Let $q$ be a prime divisor of $N$. If $q$ does not divide $MD$, then we have $M'D'=MD \neq 1$, so we may assume that $N^{\rrad}=MD$. Then, since $M'D'=MD/q=N^{\rrad}/q$ by direct computation, we have $M'D'\neq 1$ because $N^{\rrad}$ is divisible by at least two primes.
\end{proof}

Now, we can inductively define an Eisenstein series $\cE^D_{M, N}$ of weight $2$ for $\Gamma_0(N)$ as follows.

\begin{defn}\label{defn : eisenstein series}
Let $(M, D) \in \fS_0(N)$. 
\begin{enumerate}
\item
If $N=p^k$, then $\cE^{p}_{1, p^k}:=\cE_{1, p^k}$ and $\cE^{1}_{p, p^k}(z) := \cE_{p, p}(z)$.
\item
Let $q$ be a prime divisor of $N$ such that $(M', D') \in \fS_0(N')$. (Such a prime divisor always exists by Lemma \ref{lem : 5.1}.) Let $r=\mathrm{val}_q(N) \geq 1$. Assume that $\cE=\cE_{M', N'}^{D'}$ is defined. Then, we define the Eisenstein series $\cE_{M, N}^D$ as follows. 
\begin{equation}\label{eqn : ES}
\cE_{M, N}^D (z) := \begin{cases}
\cE(z)-q\cdot \cE(qz) & \text{ if } a_q(M, N, D)=1, \\
q\cdot (\cE(z)-\cE(qz)) & \text{ if } a_q(M, N, D)=q,\\
q\cdot (\cE(z)-(q+1)\cdot \cE(qz)+q\cdot \cE(q^2 z))& \text{ if } a_q(M, N, D)=0.\\
\end{cases}
\end{equation}
\end{enumerate}
Therefore the definition of $\cE_{M, N}^D$ is completed for any $(M, D) \in \fS_0(N)$ by induction.
\end{defn}

If $D=N^{\sqq}$, we frequently denote by $\cE_{M, N}$ the Eisenstein series $\cE_{M, N}^D$.
\begin{rem} \label{rem : 5.3}
If we denote by $\cE_0$ the Eisenstein series of weight $2$ and level $1$ whose $q$-expansion\footnote{This $q$-expansion is the power series $e$ on \cite[p. 78]{M77}.} is 
\begin{equation*}
1-24\sum_{m=1}^\infty \s(m) \cdot q^m,
\end{equation*}
where $\s(m)$ is the sum of the divisors of $m$, 
and we set $\cE^1_{1, 1}(z):=\cE_0(z)$. Then, the formula (\ref{eqn : ES}) above is coherent when $N'=1$ because $\cE_{p, p}=[1]_p^+(\cE_0)$. Moreover we can allow the case with $(M', D') \not\in \fS_0(N')$ and still the definition is the same.
Nevertheless, we split the definition as above not to make a confusion because $\cE^1_{1, 1}(z)$ is not a genuine modular form (cf. \cite[Defn. 2.5]{Yoo1}). Also, $\cE_{M', N'}^{D'}$ may not be a genuine modular form if $(M', D') \not\in \fS_0(N')$.
\end{rem}

\begin{prop}\label{prop : eigenvalues of eisenstein series}
Let $(M, D) \in \fS_0(N)$. Then, the ideal $\cI_{M, N}^D$ annihilates $\cE_{M, N}^D$. In other words, for any prime $p$, we have
\begin{equation*}
T_p(\cE_{M, N}^D)=a_p(M, N, D)\cdot \cE_{M, N}^D.
\end{equation*}
\end{prop}
\begin{proof}
Let $E_2(N)$ be the space of Eisenstein series of weight $2$ and level $N$. Then, the dimension of $E_2(N)$ is the number of cusps of $X_0(N)$ minus one (cf. \cite[\textsection 4.1]{DS05}).
Thus, if $N=q$ is a prime, then $E_2(q)$ is spanned by $\cE_{q, q}$. Since $E_2(q)$ is stable by the Hecke operator $T_p$ for any prime $p$, so $\cE_{q, q}$ must be an eigenform. The eigenvalue of $T_p$ on $\cE_{q, q}$ is then computed by the $q$-expansion of $\cE_{q, q}$ (cf. \cite[Prop. 5.3.1]{DS05}), which proves the claim for this case.

Note that $\cE_{1, q^2}$ is also an eigenform for all the Hecke operators $T_p$. Indeed, since $\cE_{1, q^2}=[q]_q^-(\cE_{q, q})$, by Proposition \ref{prop : the maps [p]} and Remark \ref{rem : eigenvalue} the claim follows. For $r\geq 2$, since 
\begin{equation*}
\cE_{1, q^r}=[q^{r-1}]_q \circ \cdots \circ [q^2]_q (\cE_{1, q^2}) \text{ and } 
\cE_{q, q^r}^1=[q^{r-1}]_q \circ \cdots \circ [q]_q (\cE_{q, q}),
\end{equation*}
we obtain the result when $N$ is a prime power by the same argument.

Assume that $N$ is not a prime power. Then, by Lemma \ref{lem : 5.1} there is a prime divisor $q$ of $N$ such that $(M', D') \in \fS_0(N')$. Suppose that we have the result for $\cE=\cE_{M', N'}^{D'}$. Then, by the same argument as above we obtain the result for $\cE_{M, N}^D$
because we have the following alternative description of $\cE_{M, N}^D$.
\begin{equation}\label{eqn : ES alternative}
\cE_{M, N}^D = \begin{cases}
[N'q^{r-1}]_q \circ \cdots \circ [N'q^2]_q \circ [N'q]_q \circ [N']_q^+ (\cE) & \text{ if } a_q(M, N, D)=1 \\
[N'q^{r-1}]_q \circ \cdots \circ [N'q^2]_q \circ [N'q]_q \circ [N']_q^- (\cE) & \text{ if } a_q(M, N, D)=q, \\
[N'q^{r-1}]_q \circ \cdots \circ [N'q^2]_q \circ [N'q]_q^- \circ [N']_q^+ (\cE)  & \text{ if } a_q(M, N, D)=0. \\
\end{cases}
\end{equation}
By induction, the proposition follows.
\end{proof}

\ms
\subsection{The residues of $\cE^D_{M, N}$ at various cusps}
Let $(M, D) \in \fS_0(N)$. We then compute the residues of $\cE^D_{M, N}$ at cusps as meromorphic differentials (cf. \cite[pp. 86--87]{M77} or \cite[Prop. 4.2]{Yoo3}).
\begin{thm} \label{thm : residue}
If $M \neq N^{\rrad}$, then the residue of $\cE^D_{M, N}$ at $i\infty =\vect 1 N$ is $0$, and it is $\prod_{p \mid N} (1-p)$ if $M=N^{\rrad}$.
Also, the residues of $\cE^D_{M, N}$ at any cusp of level $D$ and of level $\ell D$ are 
$$
(-1)^{\omega(D)}\cdot \cN(M, N, D) \qa (-1)^{\omega(D)+1} \cdot \cN(M, N, D),
$$
respectively, where $\ell$ is a prime divisor of $M$ not dividing $N^{\sqq}$.
\end{thm}
In the theorem, we only state the result necessary in Section \ref{sec : the index}. For a complete description of the residues of $\cE_{M, N}^D$ at all cusps, see Lemma \ref{lem : 5.8 new}.
\begin{proof}
We start with a well-known lemma in the theory of functions on compact Riemann surfaces.
\begin{lem}\label{lem : riemann surfaces}
Let $\phi$ be a non-constant holomorphic map between compact Riemann surfaces $C_1$ and $C_2$. Let $\omega$ be a meromorphic differential on $C_2$. Then, the residue of $\phi^*(\omega)$ at $x$ is the product of the ramification index of $\phi$ at $x$ and the residue of $\omega$ at $\phi(x)$.
\end{lem}
\begin{proof}
This is a local computation. On a local chart of $x$, $\phi$ sends $z$ to $z^e$, where $e$ is the ramification index of $\phi$ at $x$. If we write $\omega=f(z)dz$ on a local chart of $\phi(x)$, then $\phi^*(\omega)=f(z^e)d(z^{e})=e z^{e-1} f(z^e) dz$. Thus, if $f(z)=\sum a_n z^n$, then  
$e z^{e-1}f(z^e)=e \sum a_n z^{en+(e-1)}$. Therefore the residue of $\phi^*(\omega)$ at $x$ is $e a_{-1}$ and the claim follows because the residue of $\omega$ at $\phi(x)$ is $a_{-1}$.
\end{proof}

Using this lemma, we can compute the residues of $\cE_{M, N}^D$ inductively because two degeneracy maps on the complex points of modular curves are holomorphic (cf. Equation (\ref{eqn : alpha beta on h})). 
First, we prove the following.
\begin{lem}\label{lem : 5.7 new}
Let $d$ be a divisor of $N$, and let $\vect x d$ be a cusp of level $d$. Then, the residue of $\cE_{M, N}^D$ at a cusp $\vect x d $ is independent of $x \in \nabla(d)$. In other words, the residues of $\cE_{M, N}^D$ at all cusps of level $d$ are the same.
\end{lem}
\begin{proof}
If $N=p$ is prime, then the result follows because $X_0(p)$ has only one cusp of level $d$ for $d=1$ or $p$.
Let $\cE=\cE_{p, p}$. Then, by Lemma \ref{lem : 2.1 degeneracy maps on cusps} the ramification indices of the maps $\a_p(p)$ and $\b_p(p)$ at a cusp depend only on the level of the cusp. Thus, by Lemma \ref{lem : riemann surfaces}, the residues of $\a_p(p)^*(\cE)$ and $\b_p(p)^*(\cE)$ at a cusp $\vect{x}{p^a}$ depend only on $a$. Since $\cE_{1, p^2}=p \cdot \a_p(p)^*(\cE)-\b_p(p)^*(\cE)$ and $\cE_{p, p^2}^1=\a_p(p)^*(\cE)$, the result follows for the case $N=p^2$. Similarly, the result follows for the case $N=p^k$ for any $k\geq 3$.

If $N$ is not a prime power, choose a prime divisor $q$ of $N$ as in Lemma \ref{lem : 5.1}(3).
We use the same notation as in Lemma \ref{lem : 5.1}(3).
Let $\cE=\cE_{M', N'}^{D'}$ be an Eisenstein series of level $N'$. Assume further that the residue of $\cE$ at $\vect x d$ is independent of $x \in \nabla(d)$, so let $R_d$ denote the residue of $\cE$ at any cusp of level $d$. 
For each $t\geq 0$, we set $\g_1(t)=\a_q(N'q^t)$ and $\g_{-1}(t)=\b_q(N'q^t)$. 
For $\bd{e}=(e(r-1), \dots, e(0)) \in \{\pm 1 \}^r$, we define the map $\g_{\bd{e}} : X_0(N)=X_0(N'q^r) \to X_0(N')$ as
\begin{equation*}
\g_{\bd{e}} = \g_{e(r-1)}(r-1) \circ \cdots \circ \g_{e(0)}(0) : X_0(N'q^r) \to X_0(N'q^{r-1}) \to \cdots \to X_0(N').
\end{equation*}
Then, by Lemma \ref{lem : 2.1 degeneracy maps on cusps}, the ramification index of the map $\g_{\bd{e}}$ at a cusp $\vect x {dq^a}$ depends only on $a$ and it is a power of $q$.
So, let $q^k$ be the ramification index of the map $\g_{\bd{e}}$ at a cusp $\vect x {dq^a}$, which is independent of $x \in \nabla(dq^a)$. Then, by Lemma \ref{lem : riemann surfaces} the residue of $\g_{\bd{e}}^*(\cE)$ at a cusp $\vect x {dp^q}$ is $q^k \cdot R_d$, which is independent of $x \in \nabla(dq^a)$. Now, we define $\bd{e}(i) \in \{ \pm 1 \}^r$ by
 $\bd{e}(1)=(1, \dots, 1, 1)$, $\bd{e}(2)=(1, \dots, 1, -1)$, $\bd{e}(3)=(1, \dots, 1, -1, 1)$, and $\bd{e}(4)=(1, \dots, 1, -1, -1)$.
(Note that $\bd{e}(3)$ and $\bd{e}(4)$ are defined only for $r\geq 2$.)
Then, by the description of $\cE_{M, N}^D$ in Equation (\ref{eqn : ES alternative}), we have
\begin{equation}\label{eqn : 5.3}
\cE_{M, N}^D=\begin{cases}
\g_{\bd{e}(1)}^*(\cE)-\g_{\bd{e}(2)}^*(\cE) & \text{ if } a_q(M, N, D)=1, \\
q\cdot \g_{\bd{e}(1)}^*(\cE)-\g_{\bd{e}(2)}^*(\cE) & \text{ if } a_q(M, N, D)=q,\\
q\cdot\g_{\bd{e}(1)}^*(\cE)-q\cdot\g_{\bd{e}(2)}^*(\cE)-\g_{\bd{e}(3)}^*(\cE)+\g_{\bd{e}(4)}^*(\cE) & \text{ if } a_q(M, N, D)=0.
\end{cases}
\end{equation}
Thus, the residue of $\cE_{M, N}^D$ at a cusp $\vect x {dq^a}$ is independent of $x \in \nabla(dq^a)$ because that of each summand is.
Therefore by induction, the desired result follows.
\end{proof}

We introduce the following notation. For a prime $q$ and $0\leq a \leq r$, we set
\begin{equation*}
\cX_q(a, r):=\begin{cases}
q^{r-1}(q-1) & \text{ if } a =0 \\
q^{\fr(a)} (1-q) & \text{ if } a \geq 1,
\end{cases}
\quad
\cY_q(a, r):=\begin{cases}
q^{r-1}(q^2-1) & \text{ if } a =0 \\
0 & \text{ if } a \geq 1,
\end{cases}
\end{equation*}
and
\begin{equation*}
\cZ_q(a, r):=\begin{cases}
q^{r-2}(q^2-1)(q-1) & \text{ if } a =0 \\
q^{r-2}(1-q^2) & \text{ if } a =1\\
0 & \text{ if } a \geq 2.
\end{cases}
\end{equation*}
Here, we set $\fr(a)=\max \{0, r-2a\}$ as in Lemma \ref{lem : 4.2}. Note that 
\begin{equation}\label{eqn : XYZ comparing}
\cX_q(0, r) = \cX_q(r), ~~\cY_q(0, r)=\cY_q(r)~~\text{ and } \cZ_q(1, r) = -\cZ_q(r).
\end{equation}

\ms
Now, we are going to compute the residue of $\cE_{M, N}^D$ at any cusp.
\begin{lem}\label{lem : 5.8 new}
For a divisor $\d$ of $N$, let $a_p(\d)=\mathrm{val}_p(\d)$ and $e(p)=\mathrm{val}_p(N)$.
Then, the residue of $\cE_{M, N}^D$ at any cusp of level $\d$ is 
\begin{equation*}
\prod_{p \mid M} \cX_p(a_p(\d), e(p)) \cdot \prod_{p \mid \frac{N^{\rrad}}{MD}} \cY_p(a_p(\d), e(p)) \cdot \prod_{p \mid D } \cZ_p(a_p(\d), e(p)).
\end{equation*}
\end{lem}
\begin{proof}
For a divisor $\d$ of $N$, let $\Res_{\d}(\cE_{M, N}^D)$ be the residue of $\cE_{M, N}^D$ at any cusp of level $\d$. (By Lemma \ref{lem : 5.7 new} it is independent of a choice of a cusp of level $\d$.)

First, let $N=p$ be prime. Then, the residues of $\cE_{p, p}$ at $\vect 1 1 =0$ and at $\vect 1 p =i\infty$ are $p-1$ and $1-p$, respectively (cf. \cite[chap II, \textsection 5]{M77}). 
Since the $\a_p(p^t)$ is ramified at a cusp $\vect x {p^a}$ (of ramification index $p$) if and only if $a\leq t/2$ by Lemma \ref{lem : 2.1 degeneracy maps on cusps}, the residue of $\a_p(p)^*(\cE_{p, p})$ at any cusp of level $1$ (resp. $p$ and $p^2$) is 
$p(p-1)$ (resp. $1-p$ and $1-p$). Similarly, the residue of $\b_p(p)^*(\cE_{p, p})$ at any cusp of level $1$ (resp. $p$ and $p^2$) is $p-1$ (resp. $p-1$ and $p(1-p)$). Therefore we have
\begin{equation*}
\Res_{p^a}(\cE_{p, p^2}^1)=\cX_p(a, 2) \qa \Res_{p^a}(\cE_{1, p^2})=\cZ_p(a, 2).
\end{equation*}
Analogously (via the maps $\a_p(p^t)$ for $t\geq 2$), for $r\geq 2$ we have
\begin{equation*}
\Res_{p^a}(\cE_{p, p^r}^1)=\cX_p(a, r) \qa \Res_{p^a}(\cE_{1, p^r})=\cZ_p(a, r).
\end{equation*}

Next, assume that $N$ is not a prime power and choose $q$ as in Lemma \ref{lem : 5.1}(3). 
We use the same notation as in Lemma \ref{lem : 5.1}(3). Let $\cE=\cE_{M', N'}^{D'}$. For a divisor $d$ of $N'$, let $R_d=\Res_{d}(\cE)$.
We also use the notation in the proof of Lemma \ref{lem : 5.7 new}. 
Let $q^{r(i, a)}$ be the ramification index of $\g_{\bd{e}(i)}$ at any cusp of level $dq^a$ over some cusp of level $d$ of $X_0(N')$.
By Lemma \ref{lem : 2.1 degeneracy maps on cusps}, we have
\begin{equation*}
\begin{cases}
r(1, a)=\fr(a) & \text{ for any } r\geq 1, \\
r(2, 0)=r-1 \text{ and } r(2, a)=\fr(a)+1 &\text{ for any } a \geq 1 \text{ and } r\geq 1,\\
r(3, 0)=r-1 \text{ and } r(3, a)=\fr(a)+1 &\text{ for any } a \geq 1 \text{ and } r\geq 2,\\
r(4, 0)=r-2, ~r(4, 1)=\fr(1) \text{ and } r(4, a)=\fr(a)+2 &\text{ for any } a \geq 2 \text{ and } r\geq 2.
\end{cases}
\end{equation*}
By Lemma \ref{lem : riemann surfaces} the residue of $\g_{\bd{e}(i)}^*(\cE)$ at any cusp of level $dq^a$ is $q^{r(i, a)} \cdot R_d$. 
By Equation (\ref{eqn : 5.3}), we have
\begin{equation*}
\Res_{dq^a}(\cE_{M, N}^D)=\begin{cases}
\cX_q(a, r)\cdot R_d & \text{ if } a_q(M, N, D)=1, \\
\cY_q(a, r)\cdot R_d & \text{ if } a_q(M, N, D)=q, \\
\cZ_q(a, r)\cdot R_d & \text{ if } a_q(M, N, D)=0.
\end{cases}
\end{equation*}
This computation proves the lemma by induction.  
\end{proof}

\begin{rem} \label{rem : 5.9 new}
If $(M', N')\not\in \fS_0(N')$, then the sum of the residues of $\cE_{M', N'}^{D'}$ at all cusps is not zero, so $\cE_{M', N'}^{D'}$ cannot be a modular form on $X_0(N')$. However, the above computation still works ``formally'' (cf. Remark \ref{rem : 5.3}).
In other words, Lemma \ref{lem : 5.8 new} also holds true even for $N'=1$, by setting the residue of $\cE_{1, 1}^1$ at $\vect 1 1 =i\infty$ is $1$.
\end{rem}

Now, Theorem \ref{thm : residue} easily follows by taking $d=N$, $D$ or $\ell D$. Indeed, if $M\neq N^{\rrad}$, then there is a prime divisor $p$ of $N$ such that $a_p(M, N, D)=p$ or $0$. If $a_p(M, N, D)=0$ then $e(p) \geq 2$. Thus, either $\cY_p(e(p), e(p))=0$ or $\cZ_p(e(p), e(p))=0$. Therefore the residue of $\cE_{M, N}^D$ at the cusp $i\infty=\vect 1 N$ is $0$. If $M=N^{\rrad}$, then $\cX_p(e(p), e(p))=1-p$ and hence the residue of $\cE_{M, N}^D$ at $i\infty$ is $\prod_{p \mid N} (1-p)$. If $d=D$, then $a_p(d)=1$ for all prime divisors $p$ of $D$, and $a_p(d)=0$ otherwise. Thus, the first term is $\prod \cX_p(0, e(p))$, the second is $\prod \cY_p(0, e(p))$ and the last is $\prod \cZ_p(1, e(p))$, which proves the result by Equation (\ref{eqn : XYZ comparing}). If $d=\ell D$ then the first term is $\prod_{p \neq \ell} \cX_p(0, e(p)) \cdot \cX_\ell(1, 1)=-\prod_{p \mid M}\cX_p(0, e(p))$ and the other terms are the same as the case $d=D$. Thus, the result follows.
\end{proof}

\ms
\section{The index of $\cI_{M, N}^D$}\label{sec : the index}
In this section, we prove the following theorem.
\begin{thm}\label{thm : the index}
Let $(M, D) \in \fS_0(N)$ and 
let $\ell$ be a prime not dividing $N^\sqq$. If $\ell \geq 3$, then  
$$
\T(N)/{\cI_{M, N}^D} \otimes_{\Z} \Z_{\ell} \simeq (\zmod n) \otimes _{\Z} \Z_{\ell},
$$
where $n$ is the order of $\cC_{M, N}^D$.
If $\ell=2$, we suppose further that $N^{\rrad}/M$ is an odd integer greater than $1$.  Then, we have
$$
\T(N)/{\cI_{M, N}^D} \otimes_{\Z} \Z_{2} \simeq (\zmod n) \otimes _{\Z} \Z_{2}.
$$
\end{thm}
\begin{proof}
If $N$ is squarefree, this is done by \cite[Th. 1.4]{Yoo3}. Thus, we assume that $N^{\sqq} \neq 1$. 
We basically follow the argument in Section 5 of \textit{loc. cit.}

First, by the same argument as in \cite[chap II, Prop. 9.7]{M77}, $\T(N)/{\cI_{M, N}^D} \simeq \zmod m$ for some $m \geq 1$. More specifically, the composition $\Z \inj \T(N) \surj \T(N)/{\cI_{M, N}^D}$
is surjective because all the Hecke operators are congruent to integers modulo $\cI_{M, N}^D$. Using the duality between cusp forms and the Hecke ring, there is a cuspidal eigenform $f$ of weight $2$ and level $N$ whose $q$-expansion is $\sum_{n\geq 1} (T_n ~\mod {\cI_{M, N}^D})\cdot q^n$. If the above composition is injective, then  $\T(N)/{\cI_{M, N}^D} \simeq \Z$ and hence 
the eigenvalue of $T_p$ of $f$ is $1+p$, which is bigger than $2\sqrt{p}$ if $p$ is large enough. This is a contradiction because it violates the Ramanujan-Petersson bound. Therefore we have $\T(N)/{\cI_{M, N}^D} \simeq \zmod m$ for some $m\geq 1$.

Next, since $\br {\cC_{M, N}^D}$ is annihilated by $\cI_{M, N}^D$ by Theorem \ref{thm : hecke action on C_M, N}, there is a natural surjection
$$
\zmod m \simeq \T(N)/{\cI_{M, N}^D} \twoheadrightarrow \End (\br {\cC_{M, N}^D}) \simeq \zmod n.
$$
Therefore $n$ divides $m$.

Last, let $\ell^{a(\ell)}$ and $\ell^{b(\ell)}$ be the exact powers of $\ell$ dividing $m$ and $n$, respectively. Let $\ell$ be a prime satisfying all above assumptions above. 
Then, it suffices to prove that $a(\ell) \leq b(\ell)$.
If $a(\ell)=0$, then nothing to prove and hence we further assume that $a(\ell)\geq 1$. 
Let $f$ be a cusp form of weight $2$ for $\Gamma_0(N)$ over $\T(N)/\cI \simeq \zmod {\ell^{a(\ell)}}$ whose $q$-expansion is 
$$
\sum_{n\geq 1} (T_n ~\mod {~\cI}) \cdot q^n,
$$
where $\cI:=(\ell^{a(\ell)}, ~\cI_{M, N}^D)$. We use Theorem \ref{thm : residue}.

\begin{itemize}
\item \textbf{Case 1 :}
$M=N^{\rrad}$. Since we assume that $N^{\rrad}/M>1$ if $\ell=2$, this case cannot occur for $\ell=2$. Therefore we further assume that $\ell$ is an odd prime. By our assumption, for any prime divisor $p$ of $N^{\sqq}$, we have $\ell\neq p$ and hence we can use Lemma \ref{lem : 6.2} below. Applying it for all prime divisors $p$ of $N^{\sqq}$, we then obtain
$$
\T(N)/{\cI_{M, N}^{1}} \otimes \Z_{\ell} \simeq \T(M)/{\cI_{M, M}^1} \otimes \Z_\ell.
$$
Note that the order of $\cC_{M, N}^{1}$ is equal to the order of $\cC_{M, M}^1$ by Corollary \ref{cor : the order of C_M, N}. Therefore by \cite[Th. 1.4]{Yoo3} we have $\T(M)/{\cI_{M, M}^1} \otimes \Z_\ell \simeq \zmod {\ell^{b(\ell)}}$, which implies $a(\ell)=b(\ell)$, as desired.

\item \textbf{Case 2 :} $M\neq N^{\rrad}$ and $\ell$ does not divide $N^{\rrad}/M$.
Then the $q$-expansion of 
$$
g=24f+\cE_{M, N}^D \pmod {24\ell^{a(\ell)}}
$$ 
is zero (Theorem \ref{thm : residue}). Therefore by the $q$-expansion principle \cite[\textsection 1.6]{Ka73}, $g$ is identically zero on the connected component $C$ of $X_0(N)_{\F_\ell}$ containing the cusp $i\infty=\vect 1 N$. Hence the residue of $\cE_{M, N}^D$ at any cusp contained in the component $C$ must be divisible by $24\ell^{a(\ell)}$ because $24f$ is a cusp form. By our assumption, either $\ell$ does not divide $N$, or $\ell$ divides $M$ but it does not divide $N^{\sqq}$.
\begin{itemize}
\item
If $\ell$ does not divide $N$, then $X_0(N)_{\F_\ell}$ is connected. So, every cusp belongs to the component $C$.
\item
If $\ell$ divides $M$ but it does not divide $N^{\sqq}$, then $X_0(N)_{\F_\ell}$ has two components, and all the cusps of level $d\ell$ with $(d, \ell)=1$ belong to the component\footnote{In the proof of Lemma \ref{lem : 2.1 degeneracy maps on cusps}, we show that if $p$ exactly divide $N$, the Atkin-Lehner involution $w_p$ on $X_0(N)$ swaps the cusps of level $d$ and those of level $dp$. The operator $w_p$ also swaps two components as well.}  $C$ (cf. \cite[Ch. V]{DR73}).
\end{itemize}
Thus, in both cases the cusp $\vect 1 {\ell D}$ belongs to the component $C$. 
Since, the residue of $\cE^D_{M, N}$ at the cusp $\vect 1 {\ell D}$ is $(-1)^{\omega(D)+1} \cdot \cN(M, N, D)$ by Theorem \ref{thm : residue}, $24\ell^{a(\ell)}$ divides $\cN(M, N, D)$. In other words, $\ell^{a(\ell)}$ divides (the numerator of) $\frac{\cN(M, N, D)}{24}$. Note that by our assumption $h(M, N, D)=1$ and the exact powers of $\ell$ dividing $\cN(M, N, D)$ and $\cN(M, N/B, D)$ are the same, where $B=\prod_{p \mid \gcd(M, N^{\sqq})} p^{e(p)-1}$ and $e(p)=\mathrm{val}_p(N)$ as in Corollary \ref{cor : the order of C_M, N}.
Thus, the result follows by Theorem \ref{thm : order of C_M, N} and Corollary \ref{cor : the order of C_M, N}.

\item \textbf{Case 3 :} $\ell$ divides $N^\sqf$ but it does not divide $2M$. Then, $T_\ell- \ell \in \cI$ and $\ell$ exactly divides $N$. Thus, $\fm:=(\ell,~\cI)$ is not $\ell$-new by Lemma \ref{lem:p^2 U_p eigenvalue}.
By the same argument as in the proof of Theorem 5.2 in \cite{Yoo3}, we have
$$
\T(N)/{\cI} \simeq \T(N/\ell)/{(\ell^{a(\ell)}, ~\cI_{M, N/\ell}^D)} \twoheadleftarrow \T(N/\ell)/{\cI_{M, N/\ell}^D}.
$$
Since $\ell$ does not divide $\ell^2-1$, the exact powers of $\ell$ dividing the orders of $\cC_{M, N}^D$ and $\cC_{M, N/\ell}^D$ are the same. Since $\ell$ does not divide $2(N/\ell)$, by Cases 1 and 2 above we have
$\T(N/\ell)/{\cI_{M, N/\ell}^D} \otimes \Z_\ell \simeq \zmod {\ell^{b(\ell)}}$, and therefore $a(\ell)\leq b(\ell)$.

\end{itemize} 
\end{proof}

\begin{lem}\label{lem : 6.2}
Let $p$ be a prime divisor of $N^{\sqq}$ and let $r=\mathrm{val}_p(N)\geq 2$. 
Let $M=N^{\rrad}$ and $D=1$.
Then, 
for a prime $\ell$ different from $p$, we have an isomorphism 
$$
\T(N)/{\cI_{M, N}^1} \otimes \Z_\ell \simeq \T(N/{p^{r-1}})/{\cI_{M, N/{p^{r-1}}}^1} \otimes \Z_\ell.
$$
\end{lem}
\begin{proof}
For ease of notation, let $\T=\T(N)$, $\T^0=\T^{p\hyp\old}$ and $\T^1=\T(N/p)$. 
Let $\cI=\cI_{M, N}^1$ and $\cJ=\cI_{M, N/p}^{1}$. Let $\fm=(\ell, ~\cI)$ and $\fn =(\ell, ~\cJ)$.
Let $\T_\ell=\T \otimes \Z_\ell$ and $\T_{\fa}$ the completion of $\T$ at a maximal ideal $\fa$, i.e.,  
$\T_{\fa} = \lim\limits_{\leftarrow n} \T/{\fa^n}$. 

First, let $R$ be the common subring of $\T^0$ and $\T^1$ as in Section \ref{sec:old and new}.
By the matrix relation (\ref{eqn : matrix relation}), the first component
$\alpha_p(N/p)^*(J_0(N/p))$ is a $\T^0$-submodule of $J_0(N)_{p\hyp\old}$, and hence by the restriction, $T_p$ and $T_n$ in $\T^0$ map to $\tau_p$ and $T_n$ in $\T^1$, respectively. (Here, $n$ is a positive integer prime to $p$.)
Since $\alpha_p(N/p)^*$ is injective, we have a natural surjection
$\T^0 \twoheadrightarrow \T^1$. Since $T_p \not\in \fm$, we have 
$\T^0_{\fm} \simeq \T^1_{\fn}$.

Next, since $\T_\ell$ and $\T^1_\ell$ are ($\ell$-adically) complete and semi-local, they decompose into local rings \cite[Prop. 5]{Na50}:
$$
\T_\ell \simeq \prod_{\substack{\ell \in \fa \\ ~\fa \text{ maximal ideal }}} \T_{\fa} \quad\text{and}\quad 
\T^1_\ell \simeq \prod_{\substack{\ell \in \fb \\ ~\fb \text{ maximal ideal }}} \T^1_{\fb}.
$$
Therefore, we have
$$
\T_\ell/\cI \simeq \prod_{\substack{\ell \in \fa \text{ and } \cI \subseteq \fa \\ \fa \text{ maximal ideal }}} \T_{\fa}/\cI =\T_{\fm}/\cI \quad\text{and}\quad
\T^1_\ell/\cJ \simeq \prod_{\substack{\ell \in \fb \text{ and } \cJ \subseteq \fb \\ \fb \text{ maximal ideal }}} \T^1_{\fb}/\cJ = \T^1_{\fn}/\cJ
$$
because $\fm=(\ell,~\cI)$ and $\fn=(\ell,~\cJ)$.

Last, since $T_p \not\in \fm$, $\fm$ is not $p$-new by Lemma \ref{lem:p^2 U_p eigenvalue}. 
Thus, $\T_\fm \simeq \T^0_\fm$, and the latter is isomorphic to $\T^1_{\fn}$. By taking quotients, we have the following.
\small
$$
\T(N)/\cI_{M, N}^1 \otimes \Z_\ell \simeq \T_\ell/\cI \simeq \T_\fm/\cI \simeq \T^0_\fm/\cI \simeq  \T^1_{\fn}/\cJ\simeq \T^1_{\ell}/\cJ \simeq \T(N/p)/\cI_{M, N/p}^{1} \otimes \Z_\ell.
$$
\normalsize
Since this is true for any $r\geq 2$, we have 
$$
\T(N)/\cI_{M, N}^1 \otimes \Z_\ell \simeq \T(N/p)/\cI_{M, N/p}^{1} \otimes \Z_\ell \simeq \cdots \simeq \T(N/{p^{r-1}})/{\cI_{M, N/{p^{r-1}}}^1}\otimes \Z_\ell,
$$
as claimed.
\end{proof}

\bibliographystyle{annotation}

\begin{thebibliography}{99}


\bibitem{AL70} A.O.L. Atkin and J. Lehner, \emph{Hecke operators on $\Gamma_0(m)$}, Math. Ann. \textbf{185} (1970), 134--160.










    













\bibitem{DR73} P. Deligne and M. Rapoport, \emph{Les sch\'emas de modules de courbes elliptiques}, Modular functions of one variable II, Lecture Notes in Math., Vol. \textbf{349} (1973), 143--316.


\bibitem{DS05}  F. Diamond and J. Shurman, \emph{A first course in modular forms}, Graduate Text in Math., Vol. \textbf{228}, Springer (2005). 


\bibitem{Dr73} V. Drinfeld, \emph{Two theorems on modular curves}, Funct. Anal. Appl., \textbf{7} (1973), 155--156. 






\bibitem{FJ95} G. Faltings and B. W. Jordan, \emph{Crystalline cohomology and $\GL(2, \Q)$}, Israel J. Math. \textbf{90} (1995), 1--66.







\bibitem{Ka73} N. Katz, \emph{$p$-adic properties of modular schemes and modular forms}, Modular functions of one variable III, Lecture Notes in Math., Vol. \textbf{350} (1973), 69--190.





\bibitem{Li75} G. Ligozat, \emph{Courbes modulaires de genre 1}, Bull. Soc. Math. France, M\'emoire, tome \textbf{43} (1975), 5--80.

\bibitem{Lg95} S. Ling, \emph{Shimura subgroups and degeneracy maps}, J. Number theory, \textbf{54}(1) (1995), 39--59.

\bibitem{Lg97} S. Ling, \emph{On the $\Q$-rational cuspidal subgroup and the component group of $J_0(p^r)$}, Israel J. Math. \textbf{99} (1997), 29--54.


\bibitem{Lo95} D. Lorenzini, \emph{Torsion points on the modular Jacobian $J_0(N)$}, Compos. Math. \textbf{96} (1995), 149--172. 

\bibitem{Ma72} Y. Manin, \emph{Parabolic points and zeta functions of modular curves (in Russian)}, Izv. Akad. Nauk SSSR Ser. Mat., \textbf{36} (1972), 19--66. Translation in Math USSR-Izv, \textbf{6} (1972), 19--64.

\bibitem{M77} B. Mazur, \emph{Modular curves and the Eisenstein ideal}, Publ. Math. Inst. Hautes \'Etudes Sci., tome \textbf{47} (1977), 33--186.

\bibitem{M78} B. Mazur, \emph{Rational isogenies of prime degree}, Invent. Math.  \textbf{44} (1978), 129--162.

\bibitem{MR91} B. Mazur and K. Ribet, \emph{Two-dimensional representations in the arithmetic of modular curves}, Courbes modulaires et courbes de Shimura (Orsay, 1987/1988), Ast\'erisque, No. \textbf{196-197} (1991), 215--255.


\bibitem{Na50} M. Nagata, \emph{On the theory of semi-local rings}, Proc. Japan Acad., Vol. \textbf{26}, no. 5 (1950), 131--140.

\bibitem{Og73} A. Ogg, \emph{Rational points on certain elliptic modular curves}, Proc. Sympos. Pure Math., \textbf{24}, AMS, Providence, R. I. (1973), 221--231.



\bibitem{Oh14} M. Ohta, \emph{Eisenstein ideals and the rational torsion subgroups of modular Jacobian varieties II}, Tokyo J. Math., Vol. \textbf{37}, no. 2 (2014), 273--318. 





\bibitem{R83} K. Ribet, \emph{Congruence relations between modular forms}, Proceeding of the ICM, Vol. \textbf{1}, \textbf{2} (Warsaw, 1983) (1983), 503--514.





\bibitem{R90} K. Ribet, \emph{On modular representations of $\Gal(\overline \Q/{\Q})$ arising from modular forms}, Invent. Math.  \textbf{100}, no. 2 (1990), 431--476.






\bibitem{RS97} K. Ribet and W. Stein, \emph{Lectures on modular forms and Hecke operators}, preprint available at \url{https://wstein.org/books/ribet-stein/main.pdf}.

\bibitem{RY14} K. Ribet and H. Yoo, \emph{On the kernel of an Eisenstein prime at square-free levels}, preprint.










\bibitem{Ste85} G. Stevens, \emph{The cuspidal group and special values of $L$-functions}, Trans. Amer. Math. Soc., Vol. \textbf{291} (1985), 519--550.










\bibitem{Wi95} A. Wiles, \emph{Modular elliptic curves and Fermat's Last Theorem}, Ann. of Math.,  \textbf{142} (1995), 443--551.



\bibitem{Yoo1} H. Yoo, \emph{The index of an Eisenstein ideal and multiplicity one}, Math. Z., Vol. \textbf{282}(3) (2016), 1097--1116.


\bibitem{Yoo3} H. Yoo, \emph{On Eisenstein ideals and the cuspidal group of $J_0(N)$}, Israel J. Math. \textbf{214} (2016), 359--377.


\bibitem{Yoo5} H. Yoo, \emph{Rational torsion points on Jacobians of modular curves}, Acta Arith. \textbf{3472} (2016), 299--304.

\bibitem{Yoo7} H. Yoo, \emph{The kernel of a rational Eisenstein prime at non-squarefree level}, submitted, available at \url{https://arxiv.org/pdf/1712.01717.pdf}.

\end{thebibliography}

\end{document}